\documentclass[11pt]{amsart}
\usepackage{amsthm}
\usepackage{amssymb}
\usepackage{mathrsfs}
\usepackage{tikz}
\usepackage[margin=1in]{geometry}
\usepackage{color}
\usepackage{pifont}
\usepackage{enumitem}
\usepackage{bbm}
\usepackage{bm}
\usepackage{verbatim}
\usepackage{todonotes}
\usepackage[
	colorlinks,
	linkcolor={black},
	citecolor={black},
	urlcolor={black}
]{hyperref}
\usepackage{cite}

\newcommand{\bbC}{{\mathbb{C}}}
\newcommand{\bbD}{{\mathbb{D}}}
\newcommand{\bbN}{{\mathbb{N}}}

\newcommand{\T}{{\mathbb{T}}}
\newcommand{\bbR}{{\mathbb{R}}}

\newcommand{\bbZ}{{\mathbb{Z}}}

\newcommand{\sfE}{{\mathsf{E}}}

\renewcommand{\Im}{\operatorname{Im}}
\renewcommand{\Re}{\operatorname{Re}}

\newcommand{\Trace}{\mathrm{Tr}}

\allowdisplaybreaks

\newtheorem{theorem}{Theorem}[section]

\newtheorem{coro}[theorem]{Corollary}
\newtheorem{lemma}[theorem]{Lemma}

\theoremstyle{definition}
\newtheorem{definition}[theorem]{Definition}
\newtheorem{remark}[theorem]{Remark}

\newtheorem{conjecture}[theorem]{Conjecture}

\definecolor{purple}{rgb}{.5,0,1}
\definecolor{orange}{rgb}{1,.5,0}
\definecolor{green}{rgb}{0,.4,0}
\numberwithin{equation}{section}

\begin{document}

\title[On the Kotani--Last conjecture for the Dirac Operator]{On the Kotani--Last conjecture for the Dirac Operator}

\author[N. Davis]{Nyah Davis}
\address{N. Davis: Department of Mathematics, Rice University, Houston, TX 77005, USA}
\email{\href{mailto:nd53@rice.edu}{nd53@rice.edu}}
\author[\'I. Emilsd\'ottir]{\'Iris Emilsd\'ottir}
\address{\'I. Emilsd\'ottir: Department of Mathematics, Rice University, Houston, TX 77005, USA}
\email{\href{mailto:irist@rice.edu}{irist@rice.edu}}
\author[L. Li]{Long Li}
\address{L. Li: Department of Mathematics, Rice University, Houston, TX 77005, USA}
\email{\href{mailto:longli@rice.edu}{longli@rice.edu}}
\author[H. Liang]{Hangqi Liang}
\address{H. Liang: Department of Mathematics, Rice University, Houston, TX 77005, USA}
\email{\href{mailto:hl168@rice.edu}{hl168@rice.edu}}

%\begin{abstract}
%We prove a dichotomy of almost periodicity for a class of reflectionless Dirac operators whose spectrum satisfies certain geometrical conditions. This result generalizes the results previously obtained by A. Volberg and P. Yuditskii. 
%We construct a weakly mixing Dirac operator with a non-constant continuous potential whose spectrum is purely absolutely continuous. This result generalizes an argument of A. Avila for continuous Schr\"odinger operators. We thus disprove a conjecture of S. Kotani and Y. Last in the setting of one-dimensional Dirac operators.
%\end{abstract}

\begin{abstract}
We prove a dichotomy of almost periodicity for reflectionless one-dimensional Dirac operators whose spectra satisfy certain geometric conditions, extending work of Volberg--Yuditskii.
We also construct a weakly mixing Dirac operator with a non-constant continuous potential whose spectrum is purely absolutely continuous, adapting Avila's argument for continuous Schr\"odinger operators. In particular, we disprove the Kotani--Last conjecture in the setting of one-dimensional Dirac operators.
\end{abstract}

%\thanks{Q.Z.\ was  supported by National Key R\&D Program of China (2020YFA0713300), NSFC grant (12071232), The Science Fund for Distinguished Young Scholars of Tianjin (No. 19JCJQJC61300) and Nankai Zhide Foundation.}

\setcounter{tocdepth}{1}

%\subjclass[2000]{Primary 47A10; %%operator theory: spectrum and resolvent
%Secondary 37A20%%dynamical systems: cocycles
%}


%\begin{abstract}
%\end{abstract}

\maketitle
\tableofcontents

\section{Introduction}
We consider the one-dimensional Dirac operator acting on $L^2(\bbR,\bbC^2)$ given by
\begin{equation}\label{eq:diracOper}
\Lambda_\varphi=\begin{bmatrix}i&0\\0&-i\end{bmatrix}\frac{d}{dx}+\begin{bmatrix}0&\varphi\\\overline{\varphi}&0\end{bmatrix},
\end{equation}
where the {\it potential} $\varphi:\bbR\to\bbC$ satisfies the uniformly local integrability condition
\begin{equation}\label{eq:unLocalIntegrable}
\sup_{x\in\bbR}\int_x^{x+1}|\varphi(t)|^2dt<\infty.    
\end{equation}
This operator appears as the Lax operator in the Lax pair formalism for the cubic defocusing nonlinear Schr\"odinger equation \cite{ZS1972} and in the AKNS hierarchy \cite{AKNS1974}. Background on spectral theory in this context can be found in
\cite{BBEIM1994,GH03,GrebertKappeler2014}; for connections with Schr\"odinger operators and recent developments, see \cite{EGFR14,BLY1,EFGL, BLY2, FLLZ25} and references therein. In this paper, our focus is on almost periodicity of the potential and the resulting spectral type. 
For a detailed overview and recent advances, we refer the reader to \cite{LSS2023, LSS2024ETDS}. 
Here, we adopt the following standard notion: a function $f:\bbR\to \bbC$ is said to be {\it uniformly almost periodic} if for every $\epsilon>0$ the set 
    \[\{\tau\in\bbR:\Vert f(x+\tau)-f(x)\Vert_{\infty}<\epsilon\}\]
    is relatively dense in $\bbR,$ where $\Vert f\Vert_\infty=\sup_{x\in\bbR}|f(x)|$. 
% This definition is also called Bohr, or uniform, almost periodicity. One may replace the metric measuring $f(x+\tau),f(x)$ by some other metrics (e.g. Stepanov, Besicovitch and Weyl.) and get other notions of (Stepanov, Besicovitch and Weyl) almost periodicity\nd{unclear sentence}. \hl{If one replaces the metric $ \Vert \cdot \Vert_\infty$ used to measure the difference between $f(x+\tau)$ and $f(x)$ with other metrics, such as Stepanov, Besicovitch, or Weyl metrics, then one obtains the corresponding notions of Stepanov, Besicovitch, or Weyl almost periodicity.}
% We refer interested readers to \cite{lenz2020purepointdiffractionmean, LSS2024ETDS}.
We call the Dirac operator $\Lambda_\varphi$ \eqref{eq:diracOper} {\it almost periodic} if its potential is an almost periodic function. Such operators have been studied extensively in the setting of Jacobi matrices, CMV matrices, and continuous Schr\"odinger operators. 
A central phenomenon of interest is the presence of absolutely continuous spectrum \cite{ds, Kotani1984,  eliasson, Jitomirskaya1999Annals,BJ02, aj1, Avila1}.
While all periodic operators are known to have purely absolutely continuous spectrum, there exist limit-periodic and quasi-periodic operators with either purely or partially absolutely continuous spectrum. This motivates the following conjecture, which was made popular by \cite{KK1988, Damanik2007, Jitomirskaya2007, Simon2007}, see also \cite[Chapter 6]{DF24}.
\begin{conjecture}[Kotani--Last Conjecture]\label{conj:KLC}
    The presence of absolutely continuous spectrum implies almost periodicity of the potential.
\end{conjecture}
Earlier work provided evidence in favor of the conjecture. In the context of ergodic families, Kotani \cite{Kotani1984, Kotani1987, Kotani1989,Kotani1997} showed that the presence of absolutely continuous spectrum forces the potential to be deterministic, and Remling \cite{Remling2011Annals} proved a related result for Jacobi matrices, where potentials are eventually periodic. However, a series of counterexamples soon emerged: Avila \cite{Avila2015JAMS} disproved the conjecture for both continuum and discrete Schrödinger operators; Volberg--Yuditskii \cite{VolbergYuditskii2014} produced a disproof in the setting of reflectionless Jacobi matrices; and You--Zhou \cite{YouZhou2015IMRN} later provided a simple counterexample. The approach of Volberg--Yuditskii was subsequently adapted to continuum Schrödinger operators and CMV matrices by Damanik--Yuditskii \cite{DamanikYuditskii2016Adv}. 

Our goal here is twofold. First, we extend these disproofs to the setting of reflectionless Dirac operators whose spectrum satisfies suitable conditions. This becomes possible through recent advances in \cite{BLY2, FLLZ25}, which allow the adaptation of the method of Volberg--Yuditskii and Binder-Damanik-Goldstein-Lukić \cite{BDGL} to our framework. Second, we adapt Avila’s argument from \cite{Avila2015JAMS} to construct an explicit counterexample for the Dirac operator $\Lambda_\varphi$.
This example is particularly striking in that the potential is given by a continuous sampling of a weakly mixing flow, and hence it fails almost periodicity in the strongest possible sense: the potential has no almost periodic component whatsoever.
 
\subsection{Main Results}
Let us first state the main results, deferring the technical hypotheses to Section \ref{sec2}. Let $\mathcal{R}(\sfE)$ denote the class of {\it reflectionless} Dirac operators whose purely absolutely continuous spectrum equals an unbounded closed subset $\sfE$ of $\bbR$. Assume $\sfE$ is in generic position (see Definition~\ref{def:genericPos}) and satisfies the Widom condition (see Definition~\ref{def:WidomDomain}) and some gap length conditions \eqref{eq:finiteGapCond}, \eqref{eq:eq:gapLengthCond2}. For the precise formulation of the Direct Cauchy Theorem (DCT), see Definition~\ref{def:DCT}.

\begin{theorem}\label{thm:mainThm}
    Assume that $\varphi\in\mathcal{R}(\mathsf{E})$,  $\Omega=\bbC\setminus\mathsf{E}$ is a regular Widom domain and $\sfE$ satisfies the finite gap length conditions (c.f. \eqref{eq:finiteGapCond},\eqref{eq:eq:gapLengthCond2}), then the following dichotomy holds:
    \begin{enumerate}
        \item if the Direct Cauchy Theorem holds in $\Omega$, then every $\varphi\in\mathcal{R}(\mathsf{E})$ is almost periodic
        \item if the Direct Cauchy Theorem fails in $\Omega$, then none of $\varphi\in\mathcal{R}(\mathsf{E})$ is almost periodic.
    \end{enumerate}
\end{theorem}

\begin{coro}\label{cor.MainCor}
    There exists a Dirac operator $\Lambda_\varphi$ which is absolutely continuous, but $\varphi$ is not uniformly almost periodic.
\end{coro}
Theorem \ref{thm:mainThm} provides a criterion on almost periodicity, while Corollary \ref{cor.MainCor} needs to be supplemented with an explicit set $\sfE$. We give the construction of such set in the Appendix.

Since mixing is incompatible with almost periodicity, it would be interesting to show the existence of a weakly mixing Dirac operator with purely absolutely continuous spectrum. In the Schr\"odinger setting (both continuous and discrete), such examples were first constructed by Avila in \cite{Avila2015JAMS}. Adapting his method, we obtain an analogous result for Dirac operators. 
\begin{theorem}\label{thm:MainThm2}
    There exists a non-constant potential obtained from continuously sampling over a weakly mixing flow whose associated Dirac operator has purely absolutely continuous spectrum.
\end{theorem}

\begin{remark}
    In \cite{FLLZ25}, the third named author and his collaborators proved purely absolutely continuous spectrum for a class of quasipeirodic Dirac operators. Based on this, it is likely to extend the result of \cite{YouZhou2015IMRN} to the similar setting for Dirac operators. We do not pursue this in the current work.
\end{remark}

\subsection{Ideas of Proof}
The first claim of Theorem \ref{thm:mainThm} follows from an adapted argument of \cite{FLLZ25}. When Direct Cauchy Theorem holds on $\sfE$, the Abel map which conjugates the translation flow in the potential into a linear flow is a homeomorphism. Thus its inverse image is almost periodic. The trace formula that recovers the real part of the potential is a continuous map if the finite gap length condition holds on $\sfE$.  This proves the almost periodicity of the real part of the potential. The same argument works also for potential multiplied by $i$, recovering the imaginary part of the potential.

The proof relies on the parameterization of an isospectral family of reflectionless potentials by an infinite torus of divisors. 
This parametrization is possible only if the spectral type is purely absolutely continuous and the  gap length conditions \eqref{eq:finiteGapCond}, \eqref{eq:eq:gapLengthCond2} holds on the spectrum. 
The generalized Abel map then sends elements in the infinite torus of divisors into the character group of the fundamental group of the domain associated to the spectrum on the complex plane times an extra torus.
This map is well defined only if the Widom condition \eqref{eq:WidomCond} holds, and it is injective if and only if the Direct Cauchy Theorem (DCT) holds. In the case where DCT fails, an analysis of Volberg--Yuditskii \cite{VolbergYuditskii2014} showed that the set of irregular characters (where the generalized Abel map fails to be injective) has zero Haar measure. The pullback of the Haar measure of the character group to the isospectral family and the infinite torus of divisors can be shown to have the property that open sets are given positive weights. Finally, the rationally independent condition (generic position; Definition \ref{def:genericPos}) on the harmonic frequencies of the spectrum ensures that the flow in the character group times the unit torus conjugated by the composition of the parameterization map and the generalized Abel map is minimal and uniquely ergodic. The proof is completed by showing that the minimality and ergodicity are incompatible with uniform almost periodicity.

The proof of Theorem \ref{thm:MainThm2} follows Avila's analysis \cite{Avila2015JAMS}
of $\mathrm{SL}(2,\bbR)$ actions on the upper half plane and the dynamical realization of slow deformations of periodic potentials. We verify all the necessary ingredients in the Dirac setting, including a monotonicity property for periodic Dirac operators (see Lemma \ref{lem.rotDerivative}).

\section{Preliminaries}\label{sec2}
This section introduces the background material needed for the proofs. We begin with basic facts about $\mathrm{SL}(2,\bbR)$ actions, and then review aspects of the inverse spectral theory of reflectionless Dirac operators.  
We refer the reader to \cite{BLY1, BLY2, EFGL, FLLZ25} for further information on Dirac operators. 

\subsection{$\mathrm{SL}(2,\bbR)$ actions on $\mathbb{C}^+$}
Let $\mathbb{C}^+=\{z\in\bbC:\Im z>0\}$ denote the upper half of the complex plane. Let $A\in\mathrm{SL}(2,\bbR)$, for every $z\in\mathbb{C}^+$, let $A$ act on $z$ by 
\[A\cdot z=\frac{az+b}{cz+d}\]
where $A=\begin{bmatrix}a&b\\c&d\end{bmatrix}$. Let $\Trace(A)=a+d$ denote the trace of $A$. We call $A$ {\it elliptic} if $|\Trace(A)|<2$.
If $A$ is elliptic, there exists a unique fixed point in the upper half plane $\mathbf{u}=\mathbf{u}[A]$ such that $A\cdot \mathbf{u}=\mathbf{u}$. Clearly if $\mathbf{u}$ is fixed by $A$ then $\overline{\mathbf{u}}$ in the lower half plane is also fixed by $A.$
Moreover, when $|\Trace{A}|<2$, there exists a well defined $\mathbf{\Theta}[A]\in (0,\frac{1}{2})\cup (\frac{1}{2},1)$  and $\mathbf{B}[A]$ such that 
\[\mathbf{B}[A]A\mathbf{B}[A]^{-1}=R_{\mathbf{\Theta}[A]}.\]
Here $\mathbf{B}[A]$ may be chosen canonically up to a factor in $\mathrm{SO}(2,\bbR)$  as follows
\[\mathbf{B}[A]=\frac{1}{\Im\mathbf{u}[A]}\begin{bmatrix}1&-\Re\mathbf{u}[A]\\0&\Im \mathbf{u}[A]\end{bmatrix}\]
and $R_{\theta}=\begin{bmatrix}\cos2\pi\theta&-\sin2\pi\theta\\\sin2\pi\theta&\cos2\pi\theta\end{bmatrix}$. As functions of $A$, both $\mathbf{u}, \mathbf{B}$ are analytic.
% \subsection{Generic Position}
% For $N\in\bbN\cup\{\infty\}$, let $w_j\in\bbR,h_j\geq 0$, define the comb like domain
% \[\Pi_N=\{\omega\in\bbC:\Im w>0\}\setminus\cup_{j=0}^N\{\Re\omega=w_j:0\leq\Im\omega\leq h_j\}.\]

% By the Riemann Mapping Theorem, there exists a conformal mapping $\Phi_N:\mathbb{C}^+\to \Pi_N$ normalized by $\lim_{y\to+\infty}(\Phi_N(iy)-2iy)=0.$ Let $\sfE=\Phi_N^{-1}(\bbR)$, then $a_j=\Phi_N^{-1}(w_j-0), b_j=\Phi_N^{-1}(w_j+0)$ and we say $\sfE$ satisfies the {\it Widom's condition} if $\sum_{j=0}^Nh_j<\infty.$
% Denote $\eta_j=w_{j+1}-w_j$ and $\eta=(\eta_j)$ the {\it harmonic frequencies} of the set $\sfE.$
% \begin{definition}\label{def:genericPos}
%     We say $\sfE$ is in the generic position if its harmonic frequencies are rationally independent, that is, if $\sum_{j}c_j\eta_j=0\mod\bbZ$, then all $c_j=0.$
% \end{definition}
% We will take a second look at these notions after the introduction of the {\it generalized Abel map}.
\subsection{Reflectionless Dirac Operators} Consider the Dirac operator $\Lambda_\varphi$ with a potential $\varphi\in L^\infty(\bbR)$. Take $x=0$ we obtain two Schur functions $s_\pm(z)=s_\pm(0,z)$. By the Borg-Marchenko Theorem for Dirac operators \cite{ClarkGesztesy}, the functions $s_\pm(z)$ encode the half-line restrictions of $\Lambda_\varphi$ to $\bbR_\pm$. Notice that by general normalization of Dirac gauge, we have that $\lim_{y\to\infty}s_\pm(iy)=0.$ 

For an unbounded closed subset $\sfE=\bbR\setminus(\cup_{j}(a_j,b_j))$ such that $\overline{E}^{ess}=E$ (w.r.t. Lebesgue measure), let $\mathcal{R}(\mathsf{\sfE})$ be the set of potentials $\varphi$ whose associated Dirac operator $\Lambda_\varphi$ has purely absolutely continuous spectrum and satisfies the following conditions
\begin{enumerate}
    \item $1-s_+s_-\neq 0$ on $\bbR\setminus \sfE$
    \item the associated Schur functions satisfy $\overline{s_+(\xi+i0)}=s_-(\xi+i0)$ for Lebesgue almost every $\xi\in \sfE.$
\end{enumerate}
These two properties say that $\sigma(\Lambda)=\sfE$ and that $\Lambda_\varphi$ is \emph{reflectionless} respectively.
Let $\mathcal{S}(\sfE)$ be the set of pairs of Schur functions $(s_+,s_-)$ satisfying above conditions. Equip $\mathcal{S}(\sfE)$ with the compact open topology induced by the locally uniform convergence. It is compact under the Dirac gauge normalization.

Let $R(z)$ be the resolvent function of the Dirac operator, that is, the upper left entry of the matrix Green's function which is the integral kernel of $(\Lambda_\varphi-z)^{-1}$ for $z\in\bbC\setminus \sigma(\Lambda_\varphi)$. It is related to the Schur functions $s_\pm$ by the following relation:
\begin{equation}\label{eq:resolventFunc}
R(z) = i \frac{(1-s_+(z))(1-s_-(z))}{1-s_+(z)s_-(z)}.
\end{equation}
As a consequence of the Dirac gauge normalization, we have the asymptotic $\lim_{y\to\infty}R(iy)=i.$

\subsection{The Abel Map}
We now describe the explicit construction of $\mathcal{D}(\sfE)$ and the map from $\mathcal{S}(\sfE)$ to $\mathcal{D}(\sfE)$. 
\begin{definition}\label{def:divisor} %[Divisors]
Let $\sfE$ be as above, define $\mathcal{D}(\sfE)$  by
\begin{equation}\label{eq:divisor}
\mathcal{D}(\sfE)
=\prod_{j\in\bbN}\mathcal{I}_j,
\end{equation}
where $\mathcal{I}_j$ is a circle corresponding to gluing two copies of the $j$th gap together at the endpoints, that is,
$$
\mathcal{I}_j
=([a_j,b_j] \times \{\pm\}) / \sim,$$
where the equivalence relation is given by $(a_j,+) \sim (a_j,-)$ and $(b_j,+) \sim (b_j,-)$.
\end{definition}
Let  $\mathcal{B}:\mathcal{S}(\sfE)\to\mathcal{D}(\sfE)$ be defined as follows.
Since $R(z)$ is monotone in each spectral gap $(a_j,b_j)$, there exists at most one $\mu_j\in(a_j,b_j)$ such that $R(\mu_j)=0$. 
Suppose first that $R(\mu_j)=0$ for some $\mu_j \in (a_j,b_j)$.
Since $1-s_+(z)s_-(z)\neq 0$, this implies  either $s_+(\mu_j)=1$ or $s_-(\mu_j)=1$,
so in this case, $(\mu_j,\epsilon_j)$ is determined by $s_{\epsilon_j}(\mu_j)=1$. 
If $R(z)$ never vanishes in $(a_j,b_j)$ then either $R(z)>0$ on $(a_j,b_j)$, in which case we put $\mu_j=a_j$, or $R(z)<0$ on $(a_j,b_j)$, in which case we put $\mu_j=b_j$.
The above construction then defines the map $\mathcal{S}(\sfE)\to \mathcal{D}(\sfE)$:
\begin{equation}\label{eq:traceFormula1}\mathcal{B}: (s_+,s_-) \to D:=\{(\mu_j,\epsilon_j):j\in\bbN\}.
\end{equation}

The map $\mathcal{B}$ can also be described through the product formula for $R$.
With $\mu_j$ as in the previous paragraph, define the function $\chi$ by 
\begin{equation}\label{eq:divisor1}
\chi(\xi)=\left\{\begin{aligned}
&1/2,&& \text{if } \xi \in(a_j,\mu_j) \text{for some j}, \\
&-1/2,&& \text{if } \xi \in(\mu_j,b_j) \text{for some j}, \\
&0,&& \text{if }\xi \in\sfE.
\end{aligned}
\right.
\end{equation}
In this case, the resolvent function has the Herglotz representation on $\bbC_+$:
\begin{equation*}
R(z) 
= i |R(i)|e^{\int \! \left(\frac{1}{\xi-z}-\frac{\xi}{1+\xi^2}\right)\chi(\xi) \, d\xi}.
\end{equation*}
Under the finite-gap length condition \eqref{eq:finiteGapCond} on $\sfE$, both terms in the integrand are separately integrable. Together with the normalization $R(z) \to i$ as $z\to i\infty$, this simplifies to
\[
R(z) = i e^{\int\! \frac 1{\xi - z} \chi(\xi) \, d\xi},
\]
which yields the product representation
\begin{equation}\label{eq:resolventProd}
R(z)
= i \prod_j\sqrt{\frac{(z-\mu_j)^2}{(z-a_j)(z-b_j)}}.
%R(z)
%= i |R(i)|\prod_j\sqrt{\frac{(z-\mu_j)^2}{(z-a_j)(z-b_j)}}\left(\frac{(1+a_j^2)(1+b_j^2)}{(1+\mu_j^2)^2}\right)^{\frac{1}{4}}.
\end{equation}

We next recall the setting from \cite{BLY2} for constructing the  Abel map. Note that we only need to use the fact that the Abel map $\mathcal{A}$ between the set of divisors $\mathcal{D}(\sfE)$ and the product of the character group with the unit torus $\T=\bbR/\bbZ$ is a homoeomorphism, provided that the associated domain is Widom type and the Direct Cauchy Theorem holds. We do not use that the set of divisors $\mathcal{D}(\sfE)$ comes from the Arov gauge in their setting. In fact, we are working in the setting of Dirac gauge. The almost periodicity will come from an argument adapted from \cite{FLLZ25}.

%In order to introduce the construction of the Abel map in \cite{BLY2}, we need some assumptions on the unbounded closed set $\sfE=\bbR\setminus \cup_{j}(a_j,b_j)$ and the associated Denjoy domain $\Omega=\bbC\setminus \sfE$. Notice that $\infty\in\partial\Omega$. 
Fix a spectral gap $(a_0,b_0)$ and pick $\xi^*\in(a_0,b_0)$ for the purpose of normalization.
 Let $\pi_1(\Omega)$ be the fundamental group of the domain $\Omega=\bbC\setminus\sfE$. We write $\pi_1(\Omega)^*$ for the character group, that is, the set of $\alpha:\pi_1(\Omega)\to\T$ satisfying $\alpha(\gamma_1\circ\gamma_2)=\alpha(\gamma_2)\alpha(\gamma_1)$ for $\gamma_1,\gamma_2\in\pi_1(\Omega)$.
%\begin{definition}
%The domain $\Omega$ is called {\it Dirichlet regular} if the potential-theoretic Green's function $G(z,z_0)$  with a logarithmic pole at $z_0$ extends continuously and vanishes everywhere on $\partial\Omega$, for some (and hence for all) $z_0 \in \Omega$.
%\end{definition}

\begin{definition}\label{def:WidomDomain}
A domain $\Omega$ is called a {\it Parreau--Widom set} (or simply {\it PWS}) if, for some (equivalently, for all) $z_0\in\Omega$, one has:
\begin{equation}\label{eq:WidomCond}
\sum_{c_j:\nabla G(c_j,z_0)=0}G(c_j,z_0)<\infty.
\end{equation}
 \end{definition}
This class of surfaces was introduced by M.\ Parreau in 1958 and later, possibly independently and from a different perspective, by H.\ Widom in 1971. 
There are some other equivalent definitions, e.g.
$$\int_{0}^\infty B(\alpha,z_0)d\alpha<\infty,$$
where $B(\alpha,z_0)$ is the {\it first Betti number} of the connected domain $$\Omega(\alpha,z_0):=\{z\in\Omega:G(z,z_0)>\alpha\}$$
for each $\alpha>0.$

A fundamental result of Widom \cite{Widom1971Acta} says that the following are equivalent:
\begin{enumerate}
    \item $\Omega$ is a PWS
    \item $\mathcal{H}^\infty(\Omega,\alpha)\neq\{0\}$ for any $\alpha\in \pi_1(\Omega)^*$
    \item $\mathcal{H}^1(\Omega,\alpha)\neq\{0\}$ for any $\alpha\in \pi_1(\Omega)^*$
\end{enumerate}
where the character automorphic Hardy space $\mathcal{H}^p(\Omega,\alpha)=\{f\in \mathcal{H}^p(\Omega): f(\gamma(z))=e^{2\pi i\alpha(\gamma)}f(z)\}$.
For a detailed proof, we refer the reader to \cite[Theorem 2B of Chapter 5]{Hasumi1983}.

 Let $\sfE$ have positive capacity and  $\Omega=\bbC\setminus \sfE$  be Dirichlet regular. For $z_0\in \Omega$, let $\widetilde{G}(z,z_0)$ denote the harmonic conjugate of $G(z,z_0)$, and define the complex Green's function by
 $$
 \Phi_{z_0}(z)=e^{-G(z,z_0)-i\widetilde{G}(z,z_0)}.
 $$
 Its lift to $\bbD$ is a Blaschke product with zeros at $\{\Lambda^{-1}(z_0)\}$. 
 %Note that $\widetilde{G}(z,z_0)$ is multi-valued and $|\Phi|$ is single-valued with $$|\Phi_{z_0}(z)|=e^{-G(z,z_0)}.$$
 Since $\widetilde{G}(z,z_0)$ is multi-valued, $\Phi_{z_0}(z)$ is also multi-valued, but its modulus is single-valued, with 
 $$|\Phi_{z_0}(z)|=e^{-G(z,z_0)}.$$
 We also write $\Phi(z)=\Phi(z,i)$ for the complex Green's function with a simple zero at $i$.
 
Fix a point $\xi_* \in \bbC\setminus \sfE$ and a uniformization $\Lambda:\bbD\to\Omega$ such that $\Lambda(0)=\xi_*$. Let $\Gamma$ be the corresponding Fuchsian group acting on $\bbD$ by M\"obius transformations. Then $\Gamma \cong \pi_1(\Omega)$, and $\Lambda (z_1)=\Lambda (z_2)$ if and only if $z_2=\gamma(z_1)$ for some $\gamma \in \Gamma$.

A multi-valued function $f$ on $\Omega$ with single-valued modulus $|f|$ is lifted to a single-valued function $F$ on $\bbD$ such that 
 $F=f\circ\Lambda$. 
 There exists a homomorphism $\alpha \in \pi_1(\Omega)^*$ such that
 $$
 F \circ\gamma
 =e^{2\pi i \alpha(\gamma)}F
 \quad \forall \gamma \in \Gamma.
 $$
 We call $f$  {\it character-automorphic} with an associated character $\alpha$.
Denoting by $f\circ \gamma$ the analytic continuation of the multiplicative function $f$ along the element $\gamma$ of $\pi_1(\Omega)$, this corresponds to 
 $$f \circ \gamma 
 = e^{2\pi i \alpha(\gamma)}f \quad   \forall \gamma\in\pi_1(\Omega).$$
 In this case, we also call $f$ \emph{character-automorphic}.

 For every $\alpha\in\pi_1(\Omega)^*$, the {\it character-automorphic Hardy space} $\mathcal{H}^p(\Omega,\alpha) $ is the set of character-automorphic functions $f$ with character $\alpha$ whose lift $F = f\circ\Lambda$ belongs to $H^p(\bbD)$.

 For a Widom domain $\Omega$, the space $\mathcal{H}^2(\alpha)=\mathcal{H}^2(\Omega,\alpha)$ is a non-trivial reproducing kernel Hilbert space \cite{Widom1971Acta}. 
%The explicit construction of a linear projection $P:H^{\infty}(\bbD)\to H^\infty(\bbD,\alpha)$ for any $\alpha\in\Gamma$ can be found in \cite{Pommerenke1976}. 
For every $z_0\in\Omega$, there exists $k^\alpha_{z_0}(z)\in\mathcal{H}^2(\alpha)$ such that for any $f\in \mathcal{H}^2(\alpha)$,
 $$\langle f,k^\alpha_{z_0}\rangle=f(z_0).$$
That is, $k^\alpha_{z_0}(z)$ is the reproducing kernel corresponding to the point-evaluation functional $z_0$. 

The \emph{Smirnov class} $\mathcal{N}_+(\Omega)$ consists of analytic functions of the form $f/g$ where $f$ and $g$ are bounded analytic functions and $g$ is outer. Let $M$ denote the symmetric Martin function at the infinity of $\Omega$, and let $\Theta$ be an analytic function such that $\Im\Theta(z)=M(z)$. Denote by $\{c_j\in (a_j,b_j):j\in\bbN\}$ the set of critical points of $M$, and define the Widom function by $\mathcal{W}=\prod_j\Phi_{c_j}(z),$ which has an associated character $\beta(\mathcal{W})$. For any $f\in \mathcal{H}^1(\Omega,\beta(\mathcal{W}))$, the differential $\frac{f}{\mathcal{W}}d\Theta$ is a single-valued in $\Omega$, and moreover $\frac{f\Theta'}{\mathcal{W}}\in\mathcal{N}_+(\Omega)$.
\begin{definition}[\cite{BLY2}]\label{def:DCT}
We say that $\Omega$ satisfies the Direct Cauchy Theorem (DCT) if, for every $f\in \mathcal{H}^{1}(\Omega,\beta(\mathcal{W}))$,
\[
\oint_{\partial\Omega}\frac{f(z)}{\mathcal{W}(z)}d\Theta(z)=\oint_{\partial\Omega}\frac{f(z)\Theta'(z)}{\mathcal{W}(z)}dz=0.
\]
\end{definition}
There are various definitions of DCT, compare \cite{Hasumi1983, SY97, VolbergYuditskii2014, Yuditskii2011DCTFails,Yuditskii2020}. We are using the one that is consistent with our setting. 
% There are examples of PWS that do not satisfy DCT: see \cite{Yuditskii2011DCTFails} and references therein.
%   If $\sfE$ is homogeneous, then $\Omega = \bbC\setminus \sfE$ is known to be Dirichlet regular and satisfies DCT, but the converse is not always true; compare \cite{Yuditskii2011DCTFails, SY97}. 
 It is known that if $\Omega$ obeys DCT, then $k^\alpha_{z_0}(z_0)$ is a continuous function of $\alpha\in\pi_1(\Omega)^*$.
Define the $L^2$-normalized reproducing kernel at $z_0=i $ by $K^\alpha(z)=k^\alpha_{i }(z)/\sqrt{k^\alpha_{i }(i )}$.
Let $\sharp$ denotes the anti-linear involution $f_\sharp(z)=\overline{f(\bar{z})}$ and let $\Phi(z)$ be the complex Green's function of $\Omega$ normalized at $i $. Let $\{c_j:j\in\bbN\}$ be the set of critical points of the Green's function.
Given $\{(\mu_j,\epsilon_j):j\in\bbN\}\in \mathcal{D}(\sfE)$, the {\it canonical product} 
\begin{equation}\label{eq.prodRepRepKer}
\mathscr{K}_D(z)
=C\left\{\frac{\Phi(z)}{z-i } \frac{\Phi_{\sharp}(z)}{z+i } \prod_j \frac{z-\mu_j}{\sqrt{1+\mu_j^2}\Phi_{\mu_j}(z)} \frac{\sqrt{1+c_j^2} \Phi_{c_j}(z)}{z-c_j}\right\}^{\frac{1}{2}}\prod_j\Phi_{\mu_j}(z)^{\frac{1+\epsilon_j}{2}}
\end{equation}
arises from the decomposition \begin{equation}\label{eq:ResolventDecomp}R^{\alpha,\tau}(z)=i\frac{(1-\tau s_+^\alpha)(1-\bar{\tau}s_-^\alpha)}{1-s_+^\alpha s_-^\alpha}=\frac{i}{2}\mathscr{K}\mathscr{K}_*(z-i)(z+i)\Theta',\end{equation}
where $R^{\alpha,\tau}(z)$ is the diagonal Green's function \eqref{eq:resolventFunc} and $\Theta(z)=\widetilde{M}(z)+iM(z)$ is the complex Martin function. It is well defined if $\Omega$ is of Parreau--Widom type and satisfies the {\it finite gap condition}:
 \begin{equation}\label{eq:finiteGapCond}\sum_j(b_j-a_j)=\sum_j\gamma_j<\infty.\end{equation}

Let $M(z)$ be the Martin function of $\Omega$, in order to guarantee the condition $\int_{\xi\in\bbR\setminus\sfE}|M(\xi)|d\xi$, which is used in \cite[eq. (6.7)]{FLLZ25} for the case $k=0$, we further require $\sfE$ to satifies
\begin{equation}\label{eq:eq:gapLengthCond2}
    \sum_j\max\{|a_j|,|b_j|\}(b_j-a_j)<\infty
\end{equation}
The fact that the condition \eqref{eq:eq:gapLengthCond2} is sufficient for $\int_{\xi\in\bbR\setminus\sfE}|M(\xi)|d\xi$ is because $M(\xi)$ grows at most linearly in gaps for a Akhiezer-Levin set. 
Since $\Omega$ is multi-connected, the Martin function $\Theta$ is multi-valued and satisfies 
\begin{equation}\label{eq:charOfCompCMartin}
\Theta\circ\gamma(z)=\Theta(z)+2\pi\eta(\gamma), \gamma\in \Gamma=\pi_1(\Omega),
\end{equation}
where $\eta:\Gamma\to\bbR$ is an additive character.
It is shown in \cite[Lemma 3.12]{BLY2} that $\mathscr{K}_D$ is orthogonal to $\Phi\Phi_{\sharp}\mathcal{H}^2(\alpha-2\beta_{\Phi})$ and therefore can be written as a linear combination $$\mathscr{K}^1_D=C_1K^\alpha-C_2K^\alpha_\sharp,$$
where by $\mathscr{K}_D^C$ we make the dependence of $C$ in \eqref{eq.prodRepRepKer} explicit. In particular, denote $\mathscr{K}_D=\mathscr{K}_D^{1}$. Note that $(\mathscr{K}_D)_\sharp$ is a multiple of $\mathscr{K}_D$, it follows that $|C_1|=|C_2|=1.$ Then it follows that 
\begin{equation}\label{eq.reasonOfRotation}
\mathscr{K}_D^{\overline{C_1}}=\overline{C_1}\mathscr{K}_D=K^\alpha-\tau K^\alpha_\sharp,
\end{equation}
where $\tau=\overline{C_1}C_2.$ Intuitively $\tau$ stands for a modulation of arguments between $K^\alpha$ and $K^\alpha_\sharp$. Now, let $\alpha=\alpha(\mathscr{K}_D)$ be the character of this product.

From the above procedure we obtain the generalized Abel map
$\mathcal{A}:\mathcal{D}\to \pi^*_1(\Omega)\times \T,$
\begin{equation}
\label{eq:abelMap01}
D\mapsto (\alpha(\mathscr{K}_D),\tau).
\end{equation} 
To give the explicit formula of this map, consider a set of generators of $\pi_1(\Omega)$. Let $l_k$ be a simple loop that intersects $\bbR\setminus \sfE$ through $(a_0,b_0)$ at $\xi^*$ ``upward" and $(a_k,b_k)$ ``downward".
Then $\{l_k\}$ is a generator of the fundamental group $\pi_1(\Omega)$.
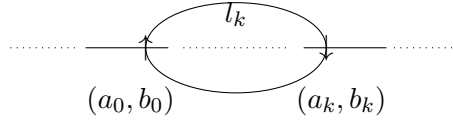
\begin{figure}
\begin{tikzpicture}
\draw[dotted] (-3,0) -- (-2.1,0);
 \draw[dotted] (2.1,0) -- (3,0);
 \draw (-2,0) -- (-0.9,0);
 \draw[dotted] (-0.7,0) -- (0.7,0);
 \draw (0.9,0)--(2,0);
 \draw (0,0) ellipse (1.2cm and 0.6cm);
\node at (-1.4,-0.7) {($a_0,b_0$)};
\node at (1.4,-0.7) {($a_k,b_k$)};
 \node at (0,0.45) {$l_k$};
 \node at (-1.2,0) {$\uparrow$};
 \node at (1.2,0) {$\downarrow$};

 \end{tikzpicture}
 \caption{A loop in $\pi_1(\Omega)$}
 \end{figure}

Let $\sfE_k$ be the subset of $\sfE$ that is inside $l_k$ and $\omega(z,F)$ is the harmonic measure of a Borel set $F\subseteq \sfE.$ We define
\begin{equation}\label{eq:Abel1}
\mathcal{A}_c[D](l_k)=\sum_j\frac{\epsilon_j}{2}\int_{a_j}^{\mu_j}\omega(d t,\sfE_k) \mod \bbZ,\quad \mathcal{A}_r(D)=\sum_j-\frac{\epsilon_j}{2}\int_{a_j}^{\mu_j}\omega(d\xi,\sfE_\infty) ~\mod\bbZ,\end{equation}
where $\omega(\lambda,F)$ is the harmonic measure of $\Omega,F\subseteq \sfE$ with respect to an interior point $\lambda\in\Omega$ and $\sfE_\infty=[\xi_*,+\infty)\cap \sfE$.

\begin{theorem}[\cite{BLY2}]\label{thm:JacobiInver}
If $\Omega=\bbC\setminus \sfE$ is a regular Widom domain that obeys DCT, $\sfE$ satisfies the gap length condition \eqref{eq:finiteGapCond}, then
the Abel map $\mathcal{A}:\mathcal{D}(\sfE)\to\pi^*_1(\Omega)\times \T$ is a homeomorphism and can be given explicitly 
\begin{equation}\label{eq:Abel2}
\alpha(l_k)=\beta_\Phi(l_k)+\mathcal{A}_{c}[D](l_k)-\mathcal{A}_{c}[D_c](l_k),
\end{equation}
where $\beta_\Phi$ is the character of the complex Green's function $\Phi$ and $D_c=\{(c_j,-1):j\in\bbN\}$. 
The rotation coordinate is obtained by \begin{equation}\label{eq:defRotationFlow}\tau=-\bar{\tau}_*^2\frac{\mathscr{K}_{D_*}(i)}{\mathscr{K}_{D_*}(-i)},\end{equation} where $D_*=\{(\mu_j,-\epsilon_j):j\in\bbN\}$ and $\tau^*$ is a unimodular constant\footnote{The explicit formula for $\tau_*$ is given by the following:
$$\tau_*=\exp\left(-i\arg\left[\frac{\Theta'(i)}{\mathcal{W}(i)}\frac{i}{\Phi'(i)}\right]\right)$$} that is independent of $D$.
\end{theorem}

Compared to the Schr\"odinger case, the Dirac Abel map carries an extra component, the rotation coordinate $\tau=\mathcal{A}_r.$ Therefore, the Abel map in our context is $\mathcal{A}=(\mathcal{A}_c,\mathcal{A}_r)$ which maps a divisor $D\in\mathcal{D}(\sfE)$ to a pair $(\alpha,\tau)\in\pi_{1}(\Omega)^*\times\T.$ 

A remarkable property of the generalized Abel map is that it conjugates the translation flow in the potential to a linear flow. A proof for finite gap spectrum can be found in \cite[Section 3]{GH03}. For the infinite gap case, we refer to \cite[Section 6]{FLLZ25} and remark that the linearization property does not depend on DCT.
\begin{theorem}[\cite{FLLZ25}]\label{thm:linAbel}
Assume that $\sfE$ satisfies the  gap length conditions \eqref{eq:finiteGapCond} and \eqref{eq:eq:gapLengthCond2}, $\Omega$ is a regular PWS. Then
    there exist $\eta\in\bbR^\infty$ and $\eta^{(1)}\in\bbR$ such that 
    $$\mathcal{A}\circ\mathcal{B}(T_\ell\varphi)=(\alpha+\ell\eta,\tau+\ell\eta^{(1)}).$$
\end{theorem}The vectors $\eta,\eta^{(1)}$ are called the {\it harmonic frequencies} of the set $\sfE$. We remark again that the condition \eqref{eq:eq:gapLengthCond2} is introduced only to force $\int_{\bbR\setminus\sfE}|M(\xi)|d\xi<\infty$ so that $\eta^{(1)}$ is well-defined.

\begin{definition}\label{def:genericPos}%[Generic position]
    We say $\sfE$ is in {\it generic position} if the sequence $(\eta,\eta^{(1)})$ is rationally independent.
\end{definition}
The genericity condition implies that $(\alpha,\tau)+(\eta,\eta^{(1)})\ell$ defines a minimal and uniquely ergodic flow on $\pi_1(\Omega)^*\times\T$ for $\ell\in\bbR$. 

The Abel map $\mathcal{A}:\mathcal{D}(\mathsf{E})\to\pi_1(\Omega)^*\times\T$ is a well-defined surjective map under the Widom condition. Moreover, if DCT holds on $\Omega$, then $\mathcal{A}$ is injective and becomes a homeomorphism \cite{BLY2}.
However, when DCT fails in $\Omega$, $\mathcal{A}$ may not be injective at some {\it irregular} $\alpha\in \pi_1(\Omega)^*$. This motivates the following definition.
 For some (hence for all) $\tau\in\partial\mathbb{D}$, the set of irregular characters is defined by
   $$ \Theta=\{\alpha\in\pi_1(\Omega)^*: \mathcal{A}^{-1}(\alpha,\tau)\text{ is not unique}\}.$$
   That is, for any $\alpha\in \Theta,$ there exist distinct $D_1,D_2\in\mathcal{D}(\mathsf{E})$ such that $\mathcal{A}(D_1)=\mathcal{A}(D_2)=(\alpha,\tau).$

The following result is a general fact of Widom domains.
\begin{theorem}[\cite{VolbergYuditskii2014}]
Assume that $\Omega=\bbC\setminus\sfE$ is a PWS. Then 
$|\Theta|_{d\alpha}=0$, where $|\cdot|_{d\alpha}$ denotes the Haar measure on $\pi_1(\Omega)^*.$
\end{theorem}

\subsection{Ergodic Dirac Operators}
Let $T_\ell(\varphi)(x)=\varphi(x-\ell)$ denote translation by $\ell$. Suppose $(\mathcal{R},T_\ell,\nu)$ is an ergodic measure-preserving system, where $\nu$ is a $T_\ell$-invariant probability measure for all $\ell\in\bbR$. For $\varphi\in\mathcal{R}$, we define $$\omega(\varphi)=\{\varphi':\text{ There exists $\ell_j\to\infty$ such that $T_{\ell_j}\varphi\to \varphi'$}\}.$$
The following result combines work of Remling \cite{Remling2011Annals,Remling2015CMP}.  
\begin{theorem}
   Let $\Sigma_{ac}$ be the essential support of the absolutely continuous spectrum of $\Lambda_\varphi,$ then
    $$\omega(\varphi)\subset\mathcal{R}(\Sigma_{ac}).$$
\end{theorem}
%Then the Poincar\'e Recurrence Theorem implies the following result; compare \cite{ConciniJohnson1987} for the simple case when $\mathsf{E}$ consists of a finite union of intervals.
By Poincaré Recurrence Theorem, one obtains an almost-everywhere recurrence statement for the translation flow; compare \cite{ConciniJohnson1987}  in the special case in which $\mathsf{E}$ is a finite union of intervals.
We will see in the next section (c.f. Corollary \ref{cor:ergodicSys}) that the set of potentials in $\mathcal{R}(\sfE)$ which map to irregular characters has zero weight with respect to the ergodic measure $d\varphi$ on $\mathcal{R}(\sfE)$. Hence we obtain:
\begin{coro}
    Assume that $\sfE$ is in the generic position and satisfies the  gap length conditions \eqref{eq:finiteGapCond}, \eqref{eq:eq:gapLengthCond2}, the domain $\Omega=\bbC\setminus\sfE$ is a regular PWS. For $d\varphi$-- a.e. $\varphi\in \mathcal{R}(\mathsf{E})$, $\Lambda_\varphi$ is reflectionless. 
\end{coro}
The natural inverse problem asks when reflectionlessness on a set $\mathsf{E}$ of positive measure precludes a singular component of the spectrum on $\sfE$.
 A set $\sfE \subseteq \bbR$ is called {\it homogeneous} (in the sense of Carleson) if
 \[\inf_{\lambda\in\sfE}\inf_{0<h<\mathrm{diam}(\sfE)}\frac{|\sfE\cap(\lambda-h,\lambda+h)|}{2h}>0,\]
  where $|\cdot|$ denotes the Lebesgue measure on $\bbR$.
 A Borel set $\sfE\subset \bbR$ is said to be {\it weakly homogeneous} if, for every $\lambda\in\sfE$,
\[\limsup_{h\to 0^+}\frac{|\sfE\cap(\lambda-h,\lambda+h)|}{2h}>0.\]

While homogeneity is sufficient for  DCT \cite{SY97}, it is not necessary \cite{Yuditskii2011DCTFails}.
A nontrivial example of a homogeneous set is a Cantor set of positive measure \cite{PeherstorferYuditskii2003}.
The following useful criterion is due to M. Sodin, based on \cite{Craig89}. 
\begin{theorem}
    Let $\mathsf{E}=[a_0,b_0]\setminus\cup_{j}(a_j,b_j)$, $\gamma_j=b_j-a_j,\gamma_0=1,$ $\eta_{jk}$ is the distance between $(a_j,b_j)$ and $(a_k,b_k),k\neq j$ and $\eta_{j0}$ is the distance between $(a_j,b_j)$ and the boundary of $[a_0,b_0]$.
    If the following condition holds
    $$\sum_{j}\frac{\gamma_j^{\frac{1}{2}}\gamma_k^{\frac{1}{2}}}{\eta_{jk}}<\infty$$
    then $\mathsf{E}$ is homogeneous.
\end{theorem}

An example of a Widom domain with DCT whose boundary is not homogeneous (but is weakly homogeneous) is given in \cite[Section 6]{Yuditskii2011DCTFails}. 
\begin{theorem}[\cite{PoltoratskiRemling2009CMP}]\label{thm:weakHomoAc}
    Suppose $\mathsf{E}$ is weakly homogeneous and $\Lambda_\varphi$ is reflectionless on $\mathsf{E}$. Then $\Lambda_\varphi$ has purely absolutely continuous spectrum on $\mathsf{E}$.
\end{theorem} 
In general, reflectionless measures may contain singular components \cite{li2026}. Theorem \ref{thm:weakHomoAc}, however,  shows that weak homogeneity excludes all singular components. One might ask whether weak homogeneity is also sufficient for the Direct Cauchy Theorem. This is not the case in general; see \cite[Section 5]{Yuditskii2011DCTFails}.

Let $\mathcal{W}$ denote the set of all Denjoy domains of Widom type. 
Let $W_{hom}\subset\mathcal{W}$ be the set of domains whose boundary is a compact homogeneous subset of $\bbR,$ $W_{DCT}\subset \mathcal{W}$ is the set of domains of $\Omega$ where DCT holds on $\Omega$ and let $W_{a.c.}\subset\mathcal{W}$ be the set of domains whose boundary supports no singular reflection measures. The following result of Yuditskii \cite{Yuditskii2011DCTFails} shows that the inclusions are strict. 
\begin{theorem}\label{thm:properSubsetDomains}
Each of the inclusions in the following is proper:
$$W_{hom}\subset W_{DCT}\subset W_{a.c.}$$
\end{theorem}
If we further define $W_{w-hom}\subset \mathcal{W}$ to be the set of domains whose boundaries are weakly homogeneous, then from the above discussions we see $W_{w-hom}\subset W_{a.c.}$ but $ W_{w-hom}\setminus W_{DCT}\neq\emptyset$. 
Whether $W_{DCT}\subset W_{w-hom}$ holds remains open to the best of our knowledge.

\subsection{Invariant Measure}
An important step in the proof of \cite{VolbergYuditskii2014} is to show that $\mathcal{R}(\mathsf{E})$ admits an ergodic measure $d\varphi$ which assigns positive weight to open sets and is invariant under translations $T_\ell(\varphi)=\varphi(\cdot-\ell).$ Retaining the three axioms of \cite{VolbergYuditskii2014} together with additional finite gap length condition \eqref{eq:finiteGapCond}, $\mathcal{R}(\mathsf{E})$ is homeomorphic to $\mathcal{D}(\mathsf{E})$. Let $d\chi$ be the pushforward measure of $d\varphi$ on $\mathcal{D}(\mathsf{E})$.

\begin{theorem}[\cite{VolbergYuditskii2014}]\label{thm:ergodicMeas}
    Let $\omega_k(z)$ be the harmonic measure of the set $\mathsf{E}_k\subset\mathsf{E}$ evaluated at $z\in\Omega$. Then for an open subset of $\mathcal{D}(\mathsf{E})$:
    $$O=\{D:\mu_{i_1}\in I_{j_1}=(a_{j_1}',b_{j_1}'),\cdots,\mu_{i_\ell}\in I_{j_{\ell}}=(a_{j_\ell}',b_{j_\ell}')\},$$
    where $[a_{j_m}',b_{j_m}']\subset (a_{j_m},b_{j_m})$, and a fixed collections of $\{\epsilon_{i_m}\}$, $1\leq m \leq \ell$,
    we have
    $$\chi(O)=\frac{1}{2^\ell}\left|\begin{aligned}
        &\omega_{j_1}(b_{j_1}')&-&\omega_{j_1}(a_{j_1}')&\cdots&\omega_{j_1}(b_{j_\ell}')&-&\omega_{j_\ell}(a_{j_\ell}')\\
        &&\vdots&&&&\vdots&\\
        &\omega_{j_\ell}(b_{j_1}')&-&\omega_{j_\ell}(a_{j_1}')&\cdots&\omega_{j_\ell}(b_{j_\ell}')&-&\omega_{j_\ell}(a_{j_\ell}')\\
    \end{aligned}\right|.$$
\end{theorem}
This is proved by a finite gap approximation, see \cite[Section 7]{VolbergYuditskii2014}.
Let $d\varphi$ be the shift invariant measure of $\mathcal{R}(\mathsf{E}).$ For a given $\varphi$, assume that $(\alpha,\tau)=\mathcal{A}\circ\mathcal{B}(\varphi),$ since $\mathcal{A}\circ\mathcal{B}(T_\ell\varphi)=(\alpha+\ell\eta,\tau+\ell\eta^{(1)})$, the Haar measure $d\alpha\times dm=\prod_{j=1}^Ndm(\alpha_j)\times dm$, where $d m(\alpha),dm$ are the Lebesgue measures on $\T$ that is shift invariant.
\begin{coro}\label{cor:ergodicSys} The system
    $(\mathcal{R}(\mathsf{E}),T_\ell,d\varphi)$ is ergodic and $|(\mathcal{A}\circ\mathcal{B})^{-1}(\{(\alpha,\tau):\alpha\in\Theta\})|_{d\varphi}=0$.
\end{coro}
\begin{proof}
Let $\pi=\mathcal{A}\circ\mathcal{B}:\mathcal{R}(\mathsf{E})\to\pi_1(\Omega)^*\times\T$ for simplicity. Let $d\varphi$ be an invariant measure of the translation, since $\pi$ is a continuous map, define $d\pi(\varphi)$ as follows
$$\int_Od\pi(\varphi)=\int_{\pi^{-1}(O)}d\varphi,$$
where $O\subset\pi_1(\Omega)^*$ is an open set.
For any $\varphi\in \mathcal{R}(\mathsf{E})$, let $\alpha$ be the character such that $\pi(\varphi)=(\alpha,\tau)$ for some $\tau$. Then by Theorem \ref{thm:linAbel}
$$\pi(T_\ell\varphi)=(\alpha+\ell\eta,\tau+\ell\eta^{(1)}),$$
where $\eta,\eta^{(1)}$ are the independent frequencies associated to $\mathsf{E}$. Consequently, for every $\alpha$,
$$\overline{\{(\alpha+\ell\eta,\tau+\ell\eta^{(1)}\}}=\pi_1(\Omega)^*\times\T.$$
Since $d\varphi$ is $T_\ell$ invariant, it follows that $d\pi(\varphi)$ is shift invariant and is the unique Haar measure, that is $d\pi(\varphi)=d\alpha\times dm.$
Let $R^\ell_\eta\alpha=\alpha+\eta\ell,R_{\eta^{(1)}}^\ell\tau=\tau+\ell\eta^{(1)}$ be the standard rotations, then $(\pi_1(\Omega)^*\times\T,R^\ell_\eta\times R^\ell_{\eta^{(1)}},d\alpha\times dm)$ is strictly ergodic. Since $|\Theta|_{d\alpha}=0$, it follows that $|\pi^{-1}(\Theta)|_{d\varphi}=0$. Then 
$\mathcal{R}(\mathsf{E})\setminus \pi^{-1}(\Theta)$ can be parametrized by $(\pi_1(\Omega)^*\setminus\Theta)\times \T.$ This is also a homeomorphism; it follows that 
$\mathcal{R}(\mathsf{E},T_\ell,d\varphi)$ is an ergodic family.
\end{proof}

As a consequence of Theorem \ref{thm:ergodicMeas}, the measure $d\chi$ on $\mathcal{D}(\mathsf{E})$ assigns positive weight to every open set in $\mathcal{D}(\mathsf{E})$. Since $\mathcal{B}$ is a homeomorphism, the invariant measure $d\varphi$ on $\mathcal{R}(\mathsf{E})$ has the same property.

\section{Proof of Theorem \ref{thm:mainThm}}
In this section, we prove Theorem \ref{thm:mainThm}. The argument parallels \cite{VolbergYuditskii2014, DamanikYuditskii2016Adv,FLLZ25}; we include it for completeness.
\begin{theorem}\label{thm.almostPeriodicDCT}Suppose that $\sfE$ satisfies the  gap length conditions \eqref{eq:finiteGapCond} and \eqref{eq:eq:gapLengthCond2}, the associated domain $\Omega=\bbC\setminus \sfE$ is a regular Widom domain.
    If DCT holds for $\Omega$, then each $\varphi\in\mathcal{R}(\mathsf{E})$ is almost periodic.
\end{theorem}
\begin{proof}
    By Theorem \ref{thm:JacobiInver}, the Abel map is a homeomorphism. By a subcase $k=0$ of \cite[Theorem 6.2]{FLLZ25}, the translation flow $x\to \varphi(\cdot+x)$ is conjugated into a linear flow in $\pi_1(\Omega)^*\times\T$, denoted by $(\alpha+\eta x,\tau+\eta^{(1)}x)$ where $(\alpha,\tau)=\mathcal{A}\circ\mathcal{B}(\varphi)$ and $\eta,\eta^{(1)}\in \bbR^\infty$ are the harmonic frequencies of $\sfE$. It follows that $\mathcal{A}^{-1}(\alpha+x\eta,\tau+x\eta^{(1)})$ is almost periodic. Thus each $\mu_j(x)$ in $D=\mathcal{A}^{-1}(\alpha+x\eta,\tau+x\eta^{(1)})$ is almost periodic. By the finite gap length condition and \cite[Lemma 5.3]{FLLZ25}, $\Re \varphi(x)=-\frac{1}{2}\sum_{j}(a_j+b_j-2\mu_j(x))$ is uniformly convergent, thus $\Re\varphi(x)$ is uniformly continuous.

    Consider the potential $i\varphi$, $\Lambda_\varphi$ is isospectral to $\Lambda_{i\varphi}$. Therefore, $i\varphi\in\mathcal{R}(\sfE)$. Let $\alpha'=\mathcal{A}\circ\mathcal{B}(i\varphi)$ and let $\alpha'+x\eta$ be the associated linear flow. Then repeating above argument gives $\Re(i\varphi)=-\Im\varphi$ is almost periodic. It follows that $\varphi$ is almost periodic. 
\end{proof}
    
We now consider the case when DCT fails for $\Omega.$
The proof can be accomplished by three steps, each of which depends on the previous step. 
\begin{lemma}\label{lem:oneBadPotential}
    If DCT fails on $\Omega$, then there exists a $\varphi\in\mathcal{R}(\mathsf{E})$ which is not almost periodic.
\end{lemma}
\begin{proof}
    Since DCT fails for $\Omega$, there exists $\alpha\in\pi_1(\Omega)^*$ for which there exist at least two elements $D_1,D_2\in\mathcal{D}(\mathsf{E})$ in  $\mathcal{A}^{-1}((\alpha,\tau))$. Let $\varphi_1,\varphi_2$ be the potentials associated to $D_1,D_2$. Then at least one of them is not almost periodic. Assume that both of them are almost periodic, pick a  regular $\beta\in\pi_1(\Omega)^*$
    for which there exists a unique $D\in\mathcal{D}(\mathsf{E})$ such that $\mathcal{A}(D)=(\beta,\gamma)$. Let $\varphi_{\beta,\gamma}$ be the potential associated with $D.$ Since  $\mathsf{E}$ is in generic position and $\mathcal{A}$ conjugates $T_\ell\varphi(\cdot)=\varphi(\cdot-\ell)$ into a linear flow in $\pi_1(\Omega)^*\times \T$ which is minimal, there exists a sequence $\ell_j$ such that $(\alpha+\ell_j\eta,\tau+\ell_j\eta^{(1)})\to(\beta,\gamma)$ as $\ell_j\to\infty.$ Almost periodicity of $\varphi_1,\varphi_2$ implies the existence of a subsequence $\{\ell_{n_k}\}$of $\{\ell_j\}$ such that 
    $$\lim\limits_{k\to\infty}\Vert T_{\ell_{n_k}}\varphi_1-\varphi_{\beta,\gamma}\Vert_\infty=\lim\limits_{k\to\infty}\Vert T_{\ell_{n_k}}\varphi_2-\varphi_{\beta,\gamma}\Vert_\infty=0.$$
    This is impossible since
    $$\begin{aligned}
    0&<\Vert\varphi_1-\varphi_2\Vert_\infty= \Vert T_{\ell_{n_k}}\varphi_1-T_{\ell_{n_k}}\varphi_2\Vert_\infty\\
    &\leq\Vert T_{\ell_{n_k}}\varphi_1-\varphi_{\beta,\gamma}\Vert_\infty+\Vert T_{\ell_{n_k}}\varphi_2-\varphi_{\beta,\gamma}\Vert_\infty\\
    &=0\quad \text{as}\quad k\to\infty.
    \end{aligned}$$
\end{proof}

\begin{lemma}\label{lem:fullMeasBadPotential}
    There exists a full measure subset $V\subset\mathcal{R}(\mathsf{E})$ such that every $\varphi\in V$ is not almost periodic.
\end{lemma}
\begin{proof}
    Let $\varphi_0$ be the potential obtained in Lemma \ref{lem:oneBadPotential}. Choose a decreasing sequence of open sets $O_n\subset\mathcal{R}(\mathsf{E})$ such that $\{\varphi_0\}=\cap_n O_n$ for all $n$ and $O_{n+1}\subset O_n$. Let $\nu$ be the ergodic measure with respect to $T_\ell$ in $\mathcal{R}(\mathsf{E})$ which assigns positive measure to open sets. Let $$V_n=\cup_{\ell\in\bbR}T_\ell O_n.$$
    Since $\nu(O_n)>0$ and $\nu$ is ergodic with respect to $T_\ell$, $\nu$ gives $V_n$ full measure. Let \begin{equation}\label{eq:fullMeasuChar}V=\cap_nV_n.\end{equation}
    Then $V$ also has full $\nu$ measure. Let's show that every $\varphi\in V$ is not almost periodic.

    Suppose there exists an almost periodic $\varphi'\in V$,  then there exists $\ell_n$ such that $\varphi'\in T_{\ell_n}O_n$ for all $n.$ In particular, by the construction of $O_n$, $$\varphi_0\in\overline{\{T_{-\ell_n}\varphi'\}}.$$
    Thus $\varphi_0$ is a uniform limit of translations of $\varphi'$, and hence $\varphi_0$ is also almost periodic. This contradicts our assumption.
    
\end{proof}
\begin{lemma}
    No $\varphi\in\mathcal{R}(\mathsf{E})$ is  almost periodic.
\end{lemma}
\begin{proof}
    Suppose there exists a $\varphi\in\mathcal{R}(\mathsf{E})$ that is almost periodic. Then any $\varphi'\in\overline{\cup_{\ell\in\bbR}\{T_\ell\varphi\}}$ is also almost periodic. Let $(\alpha,\tau)=\pi(\varphi)$, by minimality, for any $(\beta,\gamma)\in\pi_1(\Omega)^*\times\T$ there exists a sequence $\{\ell_j\}$ such that
    $$(\alpha+\ell_j\eta,\tau+\ell_j\eta^{(1)})\to(\beta,\gamma)$$
    as $j\to\infty.$
    Let us choose $\beta$ to be in the intersection of characters associated to potentials in $V$ in \eqref{eq:fullMeasuChar} with regular characters. This is a full measure subset of $\pi_1(\Omega)^*.$ It follows that $\varphi'$ is the unique potential associated to $(\beta,\gamma)$ and that $\varphi'$ is not almost periodic. This is a contradiction, since almost periodicity of $\varphi$ implies that  there exists $\{\ell_{j_k}\}$ such that 
    $$\Vert T_{\ell_{j_k}}\varphi-\varphi'\Vert_\infty\to 0$$ as $k\to\infty.$ Therefore, no $\varphi\in\mathcal{R}(\mathsf{E})$ is almost periodic.
\end{proof}
Combining the three steps, when DCT fails on $\Omega$, no $\varphi \in \mathcal{R}(\mathsf{E})$ is almost periodic. Together with Theorem \ref{thm.almostPeriodicDCT}, this proves the dichotomy in Theorem~\ref{thm:mainThm}

\section{Counter Examples to the Kotani--Last Conjecture}\label{sec6}
Yuditskii constructed an example of a weakly homogeneous set $\sfE\subset \bbR$ for which DCT fails \cite{Yuditskii2011DCTFails}. By a result of Poltoratski--Remling \cite[Corollary 2.3]{PoltoratskiRemling2009CMP}, weak homogeneity implies the absence of singular components of the reflectionless measure. Therefore, we have the following:
\begin{coro}\label{coro:counterEx}
    There exists an unbounded closed $\sfE\subset \bbR$ such that DCT fails on $\Omega=\bbC\setminus\sfE$, the spectrum of $\Lambda_\varphi,\varphi\in \mathcal{R}(\sfE)$ is purely absolutely continuous and none of $\varphi\in\mathcal{R}(\sfE)$ is almost periodic.
\end{coro}
We postpone the construction of such $\sfE$ to the Appendix. This construction was based on the examples of \cite{Yuditskii2011DCTFails} with some necessary changes to satisfy \eqref{eq:eq:gapLengthCond2} and the Widom condition \eqref{eq:WidomCond}.

The preceding counterexample is somewhat mild in that the resulting potential develops discontinuities. We therefore seek a construction of a different nature: a weakly mixing Dirac operator with purely absolutely continuous spectrum. Weak mixing is strongly incompatible with almost periodicity, whereas weaker notions (e.g. Stepanov almost periodicity) may fail to detect discontinuities. The remainder of this section is devoted to such a construction.
%In addition, we would like to construct an example of a completely different nature. Indeed, the failure of almost periodicity in Corollary \ref{coro:counterEx} is very mild, since the potential constructed develops discontinuities. However, a weaker notion of almost periodicity, such as Stepanov almost periodicity, may not be able to detect discontinuities. However, weakly mixing dynamics are strongly incompatible with almost periodicity. The rest of this section will be devoted to the construction of a weakly mixing Dirac operator with purely absolutely continuous spectrum.
Our approach generalizes Avila's argument \cite{Avila2015JAMS} from the continuous Schr\"odinger setting to Dirac operators. While most of the arguments remain the same, special care is needed in the treatment of the model-dependent part. This also illustrates the breadth of Avila’s scheme for continuous models.

We consider the Dirac operator taking the following form:
\begin{equation}
L_\varphi =  \begin{bmatrix}
	0 & -1 \\ 1 & 0
\end{bmatrix}	 \frac{d}{dx} - \begin{bmatrix} \Re \varphi(x) & \Im \varphi(x) \\  \Im \varphi(x) &  -\Re \varphi(x) \end{bmatrix}.
\end{equation}
The operator $L_\varphi$ is conjugate to $\Lambda_\varphi$ by the unitary pointwise Cayley transform, so we will continue to write $\Lambda_\varphi$ to denote the operator and identify the real-valued potentials $(\Re\varphi,\Im\varphi)$ of $L_\varphi$ with the complex-valued potential $\varphi$ of $\Lambda_\varphi$ through the standard identification  $\bbR^2$ with $\bbC$. 
The transfer matrix of the eigenvalue equation $\Lambda_\varphi U=zU$ then belongs to $\mathrm{SL}(2,\bbR)$, enabling us to utilize the analysis of $\mathrm{SL}(2,\bbR)$ action on the upper half plane that has already been understood. 
\begin{theorem}\label{thm:Main2}
There exists a weakly mixing flow $F_t$ on a solenoid $S$ and a non-constant continuous function $\Psi:S\to\bbC$ such that for the potential $\varphi_x(t)=\Psi(F_t(x))$, the Dirac operator $\Lambda_{\varphi_x}$ has purely absolutely continuous spectrum for every $x.$
\end{theorem}
%It is well known that mixing is incompatible with recurrence/almost periodicity. In particular, the weaker notion of almost periodicity in \cite{DLS25} will not help to remedy almost periodicity. 
It is well known that mixing is incompatible with recurrence and almost periodicity. 
\subsection{Time Change of Solenoidal Flows}
Let $K=\bbZ/n\bbZ$ for some positive integer $n$, and let $i:\bbZ\to K$ be the natural embedding. Set $S=K\times \bbR$, with points denoted $(x,s)\in S$. The {\it solenoidal flow} on $S$ is $F_t^S(x,s)=(x,s+t)$. 
\begin{definition}%[Time change]
    A {\it time change} of $F^S$ is a flow $F_t:S\to S$ such that $F_t(x,s)=(x,s+h(x,s,t))$ for some $C^1$ $t\mapsto h(x,s,t)$ for every $x,s$ and $w_F(x,s)=\frac{d}{dt}h(x,s,t)\vert_{t=0}$ is a continuous positive function of $(x,s)$. In this way, any continuous positive function in $S$ generates a time change.
\end{definition}
Let $S,S'$ be given and $p_{S,S'}:S\to S'$ be the natural projection. Let $F^S,F^{S'}$ be the solenoidal flow and let $F,F'$ be their respective time change. We say $F$ is $\kappa$ close to the lift of $F'$ if 
\[|\ln w_{F'}\circ p_{S,S'}-\ln w_{F}|<\kappa.\]
We say $\Psi:S\to\bbC$ is $\kappa$ close to the lift of $\Psi':S'\to\bbC$ if \[|\Psi'\circ p_{S,S'}-\Psi|<\kappa.\]
%\begin{definition}%[Projective limit]
%    An increasing sequence of compact Abelian groups is given by a sequence of compact Abelian groups $K_n$ and projections $p_{n,n'}:K_n\to K_{n'}$ for $n>n'$ such that $p_{n_2,n_3}\circ p_{n_1,n_2}=p_{n_1,n_3}$.
%    For such a sequence, there exists a compact Abelian group $K$ and a sequence of projections $p_n:K\to K_n$ such that for all $n>n'$, $p_{n'}=p_{n,n'}\circ p_n$ and $p_j$ separates points in $K$. That is, for any $x\neq y\in K$, there exists at least one index $j$ such that $p_j(x)\neq p_j(y)$. One may think of elements in $K$ as $x=(x_j)$ with $x_j\in K_j$ and $p_j(x)=x_j$ and endow $K$ with the product topology and obvious group structure. Such $K$ is called the {\it projective limit}.
%\end{definition}
\begin{definition}%[Projective limit]
    Let $(K_n)_{n\geq 1}$ be an increasing sequence of compact Abelian groups with surjective homomorphisms (projections) $p_{n,n'}:K_n\to K_{n'}$ for $n>n'$ satisfying $p_{n_2,n_3}\circ p_{n_1,n_2}=p_{n_1,n_3}$, $n_1>n_2>n_3$.
    There exists a compact Abelian group $K$ together with projections $p_n:K\to K_n$ such that for all $n>n'$ we have $p_{n'}=p_{n,n'}\circ p_n,$ and $p_j$ separates points in $K$. That is, for any $x\neq y\in K$, there exists at least one index $j$ such that $p_j(x)\neq p_j(y)$. One may think of elements in $K$ as $x=(x_j)$ with $x_j\in K_j$ and $p_j(x)=x_j$ and endow $K$ with the product topology and obvious group structure. Such $K$ is called the {\it projective limit}.
\end{definition}
In the current context, we will not distinguish $p_{n,n'}$ from $P_{K_n,K_{n'}}$. The following result will be needed.
\begin{lemma}[\cite{Avila2015JAMS}]
Let $S_n$ be an increasing sequence of solenoids with projective limit $S$. Let $\epsilon_n$ be a sequence of positive numbers converging to zero. Let $F_n$ be the time changes of the solenoidal flow $F^{S_n}$, and let $\Psi_n:S_n\to\bbC$ be continuous. Assume that for any $n>n'$, $(F_{n},\Psi_n)$ is $\epsilon_{n'}$ close to the lift of $(F_{n'},\Psi_{n'})$. Then there exists a time change $F$ of $F^S$ and a continuous function $\Psi:S\to\bbC$ such that $(F,\Psi)$ is $\epsilon_n$ close to the lift of $(F_n,\Psi_n)$ for every $n.$
\end{lemma}
\begin{definition}%[$(N,F)$--mixed] 
Let $F_t:S\to S$ be a time change of a periodic solenoidal flow. We say a time change $F_t':S'\to S'$ of a solenoidal flow is $(N,F)$--mixed if $S'$ projects to $S$ (by $p:S'\to S$) and for every $1\leq j\leq N$ there exist subsets $U_j,V_j$ of $S'$ of Haar measure larger than $\frac{1}{3}$ such that 
for each $x\in U_j$, there exists $|t|<\frac{1}{N}$ such that
\[p(F'_{t_j}(x))=F_t(p(x)).\]
For each $x\in V_j$, there exists $|t-\frac{j}{N}|<\frac{1}{N}$ such that 
\[p(F'_{t_j}(x))=F_t(p(x)).\]
\end{definition}
\begin{lemma}[\cite{Avila2015JAMS}]\label{lem:stabilityOfMixed}
    Let $F'$ be $(N,F)$--mixed. There exists $\kappa>0$ such that if $\tilde{F}$ is $\kappa$--close to the lift of $F'$, it is also $(N,F)$--mixed.

    Let $F$ be the projective limit of $F^{(n)}$. If for every $N\in\bbN$ there exists $n$ such that $F$ is $(N,F^{(n)})$--mixed, then $F$ is weakly mixing.
\end{lemma}

\subsection{Floquet theory for Dirac Operators}
We briefly recall Floquet theory for periodic Dirac operators. A potential $\varphi$ is periodic if there exists $P>0$ such that $\varphi(x+P)=\varphi(x)$. Let $T(z,x,y)$ denote the transfer matrix of the eigenvalue equation $\Lambda_\varphi U=zU$ such that $U(z,y)=T(z,x,y)U(z,x)$. The {\it monodromy matrix} is $T(z,x)=T(z,x,x+P)$ and we write $T(z)=T(z,0)$. Let $D(z)=\Trace(T(z))$ be the {\it discriminant}. 
The following result was proved in \cite{EFGL}.
\begin{theorem}[\cite{EFGL}]
    Suppose $\varphi\in C(\bbR)$ is $P$-periodic and let $T(z,x)$ and $D(z)$ denote the associated monodromy and discriminant.
    \begin{enumerate}
        \item $\mathrm{spr}(T(z,x))$ and $D(z)$ are independent of $x$.
        \item The {\it Lyapunov exponent} is given by \[L(z)=\lim\limits_{x\to\infty}\frac{1}{x}\log\Vert T(z,0,x)\Vert=\frac{1}{P}\log \mathrm{spr}(T(z,x)).\]
        \item The spectrum of $\Lambda_\varphi$ is given by 
        \[\Sigma=\{\lambda\in\bbR:D(\lambda)\in[-2,2]\}=\mathcal{Z}.\]
        \item For $\lambda\in\bbR$ such that $D(\lambda)\in (-2,2)$, there exists $B_\lambda(x)\in \mathrm{SU}(1,1)$ such that 
        \[B_\lambda(x)T(\lambda,x)B_\lambda(x)^{-1}\in K=\{\begin{bmatrix}e^{i\theta}&0\\0&e^{-i\theta}\end{bmatrix}:\theta\in [0,2\pi)\}.\]
        \item Given $R>0$, there are at most $2(\frac{T}{\pi}(R+\Vert\varphi\Vert_\infty)+1)$ spectral bands of $\Sigma$ that intersect $[-R,R].$
    \end{enumerate}
\end{theorem}
The following spectral continuity for self-adjoint operators is standard, see e.g. \cite{RS}.
\begin{lemma}\label{lem:spectralCont}
    Let $\varphi:\bbR\to\bbC$ be a bounded continuous function. Let $\varphi^{(n)}:\bbR\to\bbC$ be a sequence of uniformly bounded continuous functions such that $\varphi^{(n)}$ converges to $\varphi$ uniformly on compact sets. Let $\mu_\varphi,\mu_{\varphi^{(n)}}$ be the associated spectral measure of $\Lambda_\varphi,\Lambda_{\varphi^{(n)}}$. Then $\mu_{\varphi^{(n)}}$ converges to $\mu_\varphi$ weakly.
\end{lemma}

Let $\varphi_s(x)=\varphi(x+s)$ be the $s$-translation. The following can be directly computed.
\begin{lemma}
For almost every $\lambda$ in the spectrum, we have
\[\frac{d\mu_{\varphi_s}}{d\lambda}=\frac{1}{\Im u[\varphi](\lambda,s)}.\]
\end{lemma}
\begin{proof}
Let $m_\pm(s,z)$ be the Weyl functions associated to Weyl solutions that are square integrable at $\pm\infty$ with the boundary condition $[m_\pm(s,z),1]^\top$ at $x=s$ respectively. Then they are the fixed points on the upper and lower half planes of the monodromy matrix $T[\varphi](z,0,T)$.
Since periodic Dirac operators are reflectionless, we have $m_+(\lambda+i0)=\overline{m_-(\lambda+i0)}$ for a.e. $\lambda$. Consider the Borel transform $M(z)=\int\frac{1}{\lambda-z}\mu(\lambda)$. It follows that $M(z)=\frac{2}{m_-(z)-m_+(z)}$ (c.f. \cite{ClarkGesztesy}). Therefore, by a standard result, 
\[\frac{d\mu}{d\lambda}=\Im M(\lambda+i0)=\frac{1}{\Im m_+(\lambda+i0)}=\frac{1}{\Im u[\varphi](\lambda,s)}.\]
\end{proof}

Let $\mu_{\varphi_s,C}$ be the restriction of $\mu_{\varphi_s}$ to $\{\lambda\in\bbR:|d\mu_{\varphi_s}/d\lambda|<C\}$. We say a periodic potential $\varphi_s$ is \emph{$(\epsilon,C,M)$--uniform} if \[\mu_{\varphi_s}(-\infty,M]-\mu_{\varphi_s,C}(-\infty,M]<\epsilon\]
for every $s.$

The following result is standard for periodic potentials.
\begin{lemma}\label{lem:periodicUniform}
    For every periodic $\varphi,\epsilon,M,$ there exists a $C>0$ such that $\varphi$ is $(\epsilon,C,M)$--uniform.
\end{lemma}
Let us denote $p=\Re\varphi, q=\Im\varphi$ and $T[\varphi](\lambda,t,s)$ be the transfer matrix of the eigenvalue equation $L_\varphi f=\lambda f.$ Direct computation gives 
\begin{equation}\label{eq.deriTrans}
    \frac{d}{ds}T[\varphi](\lambda,t,s)=\begin{bmatrix}
        q&\lambda-p\\-(\lambda+p)&-q
    \end{bmatrix}T[\varphi](\lambda,t,s).
\end{equation}
The following definition in the opposite direction (i.e. {\it weakly positive}) was introduced in the proof \cite[Lemma 2.1]{Avila2015JAMS}. 

    We say that $a\in\mathrm{sl}(2,\bbR)$ is {\it weakly negative} if  
    \[\alpha(a,w):=\mathrm{det}(w,a\cdot w)\leq 0\text{ for every } w\in\bbR^2\]
    and 
    \[\alpha(a,w)<0 \text{ for some } w\in\bbR^2,\]
    where $(w,a\cdot w)$ is the matrix with columns $w$ and $a\cdot w$, the action $a\cdot w$ is the standard matrix multiplication of $\mathrm{SL}(2,\bbR)$ with vectors in $\bbR^2.$

For any $T\in \mathrm{SL}(2,\bbR),a\in \mathrm{SL}(2,\bbR)$, we have \[\alpha(T^{-1}aT,w)=\mathrm{det}(Tw,TT^{-1}a\cdot (T w))=\alpha(a,Tw),\]
and if $a$ is weakly negative then $T^{-1}aT$ is weakly negative for every $T\in \mathrm{SL}(2,\bbR)$.

\begin{lemma}\label{lem.rotDerivative}
    Let $s>t$. If $|\Trace(T[\varphi](\lambda,t,s))|<2$, then 
    \[\frac{d}{d\lambda}\mathbf{\Theta}(T[\varphi](\lambda,t,s))<0.\]
\end{lemma}
\begin{proof}
    Denote $T(\lambda,t,s)=T[\varphi](\lambda,t,s)$ for simplicity, direct computation gives 
    \[\partial_s(T(\lambda,t,s)^{-1}\partial_\lambda T(\lambda,t,s))=T(\lambda,t,s)^{-1}\begin{bmatrix}0&1\\-1&0\end{bmatrix}T(\lambda,t,s).\]

 \textbf{Claim:} If $\lambda\mapsto T(\lambda)$ is a $C^1$ analytic family of $\mathrm{SL}(2,\bbR)$ matrices and $T(\lambda)^{-1}\partial_\lambda T(\lambda)$ is weakly negative and $|\Trace(T(\lambda))|<2$, then $\partial_\lambda\mathbf{\Theta}(T(\lambda))<0.$ 
\begin{proof}[Proof of the Claim]
 Let us first prove this claim and then verify the condition that $T(\lambda)^{-1}\partial_\lambda T(\lambda)$ is weakly negative.
Since $|\Trace(T(\lambda))|<2$, there exists \[\mathbf{B}[T(\lambda)]=\frac{1}{(\Im\mathbf{u}[T(\lambda)])^{1/2}}\begin{bmatrix}1&-\Re\mathbf{u}[T(\lambda)]\\0&\Im\mathbf{u}[T(\lambda)]\end{bmatrix}\]
such that \[\mathbf{B}[T(\lambda)]T(\lambda)\mathbf{B}[T(\lambda)]^{-1}=R_{\mathbf{\Theta}[T(\lambda)]}\]
where $\mathbf{u}[T(\lambda)]$ is the unique fixed point of $T(\lambda)$ in the upper half plane.
Therefore, $\Trace(T(\lambda))=2\cos2\pi\mathbf{\Theta}[T(\lambda)]$ and 
\[\partial_\lambda\mathbf{\Theta}[T(\lambda)]=-\frac{\Trace(\partial_\lambda T(\lambda))}{4\pi\sin2\pi\mathbf{\Theta}[T(\lambda)]}.\]
Let $J=\begin{bmatrix}0&-1\\1&0\end{bmatrix}$, we can write
\[R_{\mathbf{\Theta}[T(\lambda)]}=\cos2\pi\mathbf{\Theta}[T(\lambda)] I+\sin2\pi \mathbf{\Theta}[T(\lambda)] J.\]
Let $s=T(\lambda)^{-1}\partial_\lambda T(\lambda),$ so that $\partial_\lambda \mathbf{\Theta}[T(\lambda)]=-\frac{\Trace(T(\lambda)s)}{4\pi\sin2\pi \mathbf{\Theta}[T(\lambda)]}.$
For any $u\in\bbR^2$ and any $2\times 2$ matrix $N$,
\[\mathrm{det}(u,Nu)=-\Trace(uu^\top JN).\]
If we set $v(t)=[\cos t,\sin t]^\top$, then \[\alpha(s,v(t))=\mathrm{det}(v,sv)=-\Trace(vv^\top Js),\] that is, the determinant of the matrix with column vectors $v$ and $sv.$
Since $\frac{1}{2\pi}\int_{0}^{2\pi}v(t)v^\top(t)dt=\frac{1}{2}I,$ we obtain
\[\frac{1}{2\pi}\int_{0}^{2\pi}\alpha(s,v(t))dt=-\frac{1}{2}\Trace(Js).\]
Let $\tilde{s}=\mathbf{B}[T(\lambda)]s\mathbf{B}[T(\lambda)]^{-1}$. Then $\tilde{s}$ is weakly negative whenever $s$ is.
Clearly,
\[\Trace(T(\lambda)s)=\Trace(R_{\mathbf{\Theta}[T(\lambda)]}\tilde{s})=\sin2\pi\mathbf{\Theta}[T(\lambda)]\Trace(J\tilde{s}).\]
Therefore, we have 
\[\partial_\lambda\mathbf{\Theta}[T(\lambda)]=\frac{1}{2\pi}\int_0^{2\pi}\mathrm{det}(v(t),\tilde{s}v(t))\frac{dt}{2\pi}.\]
Since $\tilde{s}$ is weakly negative, we have $\int_0^{2\pi}\mathrm{det}(v(t),\tilde{s}v(t))\frac{dt}{2\pi}<0$ and the claim follows.
\end{proof}
% Since \[\partial_\lambda R_{\mathbf{\Theta}[T(\lambda)]}=\partial_\lambda\mathbf{\Theta}[T(\lambda)]JR_{\mathbf{\Theta}[T(\lambda)]},~R_{\mathbf{\Theta}[T(\lambda)]}^{-1}JR_{\mathbf{\Theta}[T(\lambda)]}=J\] where $J=\begin{bmatrix}0&1\\-1&0\end{bmatrix}$, we have
% \[R_{\mathbf{\Theta}[T(\lambda)]}^{-1}\partial_\lambda R_{\mathbf{\Theta}[T(\lambda)]}=\partial_\lambda\mathbf{\Theta}[T(\lambda)]J.\]
% For any $v\in\bbR^2,$ $\mathrm{det}(v,Jv)=\Vert v\Vert^2,$ thus we compute
% \[\mathrm{det}(v,R_{\mathbf{\Theta}[T(\lambda)]}^{-1}\partial_\lambda R_{\mathbf{\Theta}[T(\lambda)]}v)=\partial_\lambda\mathbf{\Theta}[T(\lambda)]\Vert v\Vert^2.\]
% Therefore, for $T(\lambda)$ a rotation, the claim holds.

% Next let us consider a general $T(\lambda)^{-1}\partial_\lambda T(\lambda)\in \mathrm{SL}(2,\bbR)$ that is weakly negative.
% We compute for $v=\mathbf{B}[T(\lambda)]w\in\bbR^2$
% \[\mathrm{det}(v,\mathbf{B}[T(\lambda)](T(\lambda)^{-1}\partial_\lambda T(\lambda))\mathbf{B}[T(\lambda)]^{-1}v)=\mathrm{det}(w,(T(\lambda)^{-1}\partial_\lambda T(\lambda))w)\]
% since $\mathrm{det}(\mathbf{B}[T(\lambda)])=1.$

Now let us verify that $T(\lambda)^{-1}\partial_\lambda T(\lambda)$ is weakly negative. Notice that \[T(\lambda,t,s)^{-1}\partial_\lambda T(\lambda,t,s)=\int_{t}^sT(\lambda,t,q)^{-1}\begin{bmatrix}0&1\\-1&0\end{bmatrix}T(\lambda,t,q)dq.\]
Since the constant matrix $\begin{bmatrix}
    0&1\\-1&0
\end{bmatrix}$ is weakly negative, it follows that $T(\lambda,t,q)^{-1}\begin{bmatrix}0&1\\-1&0\end{bmatrix}T(\lambda,t,q)$ is also weakly negative.
 The set of weakly negative matrices in $\mathrm{SL}(2,\bbR)$ is an open convex cone, and so it follows that $T(\lambda,t,s)^{-1}\partial_\lambda T(\lambda,t,s)$ is weakly negative.
\end{proof}
\begin{remark}
    Lemma \ref{lem.rotDerivative} is necessary to guarantee the application of the next lemma (\cite[Lemma 2.12]{Avila2015JAMS}). Let $\tilde{\theta}(\lambda)=\int_{\bbR/\bbZ}\mathbf{\Theta}[T(\lambda,t)]dt$ where $T:J\times\bbR/\bbZ\to \mathrm{SL}(2,\bbR)$ for some closed interval $J.$ The condition $\frac{d\tilde{\theta}(\lambda)}{d\lambda}\neq 0$ is therefore guaranteed by Lemma \ref{lem.rotDerivative}. For camparison, in the Schr\"odinger case Avila proves $\frac{d\mathbf{\Theta}[T(\lambda)]}{d\lambda}>0$.
\end{remark}

\begin{lemma}[\cite{Avila2015JAMS}]\label{lem:equidistributionLem}
    Let $J$ be a closed interval of $\bbR$, and let $T:J\times\bbR/\bbZ\to \mathrm{SL}(2,\bbR)$ be a smooth function such that $|\Trace(T(\lambda,t))|<2$ for $(\lambda,t)\in J\times\bbR/\bbZ.$ Let $B(\lambda,t)=\mathbf{B}[T(\lambda,t)],\theta(\lambda,t)=\mathbf{\Theta}[T(\lambda,t)]$. Let $\tilde{\theta}(\lambda)=\int_{\bbR/\bbZ}\theta(\lambda,t)dt.$ Assume that \[\frac{d}{d\lambda}\tilde{\theta}\neq 0 \text{ for every } \lambda\in J.\] For $n\in\bbN$ and define 
    \[T^{(n)}(\lambda,t)=\prod_{j=n-1}^0T(\lambda,t+\frac{j}{n}).\]
    Then there exist functions $\tilde{\theta}^{(n)}:J\to\bbR$ such that for every measurable subset $Z\subset \bbR/\bbZ$
    \begin{equation}\label{eq:equidistribution}
        \lim\limits_{n\to\infty}|\{\lambda\in J:\tilde{\theta}^{(n)}(\lambda)\in Z\}|=|Z||J|.
    \end{equation}
    Moreover, for every $\delta>0$ we have:
    \[\lim\limits_{n\to\infty}\Vert \Trace(T^{(n)}(\lambda,t))-2\cos2\pi\tilde{\theta}^{(n)}(\lambda)\Vert_{C^0(J\times\bbR/\bbZ, \bbR)}=0.\]
    \[\lim\limits_{n\to\infty}\sup_{|\sin2\pi\tilde{\theta}^{(n)}(\lambda)|>\delta}\Vert \mathbf{\Theta}[T^{(n)}(\lambda,\cdot)]-\tilde{\theta}^{(n)}(\lambda)\Vert_{C^1(J\times\bbR/\bbZ,\bbR)}=0.\]
    \[\lim\limits_{n\to\infty}\sup_{|\sin2\pi\tilde{\theta}^{(n)}(\lambda)|>\delta}\Vert \mathbf{u}[T^{(n)}(\lambda,\cdot)]-\mathbf{u}[T(\lambda,\cdot)]\Vert_{C^1(J\times\bbR/\bbZ,\bbC)}=0.\]
\end{lemma}

\subsection{The $(\delta,n)$ Padding} Let $\varphi:\bbR/P\bbZ\to\bbC$ be a $P$--periodic potential with $\varphi(0)=0$. For $\delta>0$ and $n\in\bbN$, define $\varphi':\bbR/P'\bbZ\to\bbC, P'=2nP+\delta n$ as follows:
\begin{enumerate}
    \item Define breakpoints $a_j=\left\{
    \begin{aligned}
        &jP&0\leq j\leq n\\
        &jP+(j-n)\delta&n+1\leq j\leq 2n
    \end{aligned}
    \right.$
    \item On each $[a_j,a_j+P]$, set $\varphi'(x)=\varphi(x-a_j),$  $0\leq j\leq 2n-1.$
    \item On $[a_j+P,a_{j+1}]$ for $n\leq j\leq 2n-1$, set
    $\varphi'(x)=0$.
\end{enumerate}
The $(\delta,n)$ padding of a periodic potential is the main engine in Avila's construction.
The following lemmas show that for sufficiently small $\delta>0$ and sufficiently large $n$, the $(\delta,n)$ padding of $\varphi$ can be realized as $\varphi'=\Psi(F'_t(0))$, with  $F'_t:S'\to S'$ a continuous sampling of a time change of solenoidal flow on $S'$. Moreover, $(F',\Psi)$ is $(N,F)$--mixed and the new potential $\varphi'$ preserves the $(\epsilon,C,M)$--uniform property of $\varphi$.

The following result is a restatement of \cite[Lemma 3.4]{Avila2015JAMS}. It realizes the $(\delta,n)$--padding of a periodic potential as sampling over a periodic time change of a solenoidal flow.
\begin{lemma}\label{lem:realizeAsFlow}
    Let $F_t:S\to S$ be a $P$--periodic time change of a solenoidal flow. Let $\varphi:\bbR/P\bbZ:\to\bbC, \Psi:S\to\bbC$ be continuous functions such that $\varphi(t)=\Psi(F_t(0))$ is smooth. Assume that $\varphi(t)=0, t\in [P-\epsilon_0,P]$ for some $0<\epsilon_0<P$. Then for every $0<\delta<\epsilon_0$ and every $n\in\bbN$ the $(\delta, n)$--padding $\varphi'$ of $\varphi$ can be realized as $\varphi'(t)=\Psi'(F'_t(0))$ where $F'_t:S'\to S'$ is a $P'$ time change of a solenoidal flow, $\Psi':S'\to\bbC$ is continuous and $(F',\Psi')$ is $\frac{\delta}{\epsilon_0}$--close to a lift of $(F,\Psi)$.
\end{lemma}

The following result proved in \cite{Avila2015JAMS} is also standard and model-independent.
\begin{lemma}
    Let $F,\varphi,\Psi,\epsilon_0$ be as in the previous lemma. Then for every $\delta>0$ sufficiently small, for every $N\in\bbN$ and for $n$ sufficiently large, the $(\delta,n)$--padding $\varphi'$ of $\varphi$ takes the form $\varphi'(t)=\Psi'(F'_t(0))$, where $F'_t:S'\to S'$ is a $P'$--periodic time change of a solenoidal flow, $\Psi':S'\to S'$ is continuous, $(F',\Psi')$ is $\frac{\delta}{\epsilon_0}$--close to a lift of $(F,\Psi)$ and $(F',\Psi')$ is $(N,F)$--mixed. 
\end{lemma}

We will also need the following result to control the spectral measure. Since it is model-dependent, we provide a proof.
\begin{lemma}\label{lem:stableUniform}
    Let $\varphi:\bbR/P\bbZ$ be a continuous function with $\varphi(0)=0$. If $\varphi$ is $(\epsilon,C,M)$--uniform, then for sufficiently small $\delta>0$ and every $n\in\bbN$, the $(\delta,n)$ padding $\varphi'$ of $\varphi$ is also $(\epsilon,C,M)$--uniform.
\end{lemma}
\begin{proof}
Let $T(\lambda)=T[\varphi](\lambda)$  and $\Omega(\varphi)=\{\lambda\in\bbR:|\Trace(T(\lambda))|<2\}$. Let $J\subset \Omega(\varphi)\cap(-\infty,M]$ be a finite union of closed intervals such that 
\begin{equation}\label{eq.uniformEst}\sup_{s}\mu_{\varphi_s}(-\infty,M]-\mu_{\varphi_s,C}(J)<\epsilon_0<\epsilon.\end{equation}
For $\lambda\in\Omega(\varphi), T(\lambda)$ is elliptic and therefore $B(\lambda)=\mathbf{B}[T(\lambda)], \theta(\lambda)=\mathbf{\Theta}[T(\lambda)]$ are well-defined analytic functions. Moreover, 
\[B(\lambda)T(\lambda)B(\lambda)^{-1}=R_{\theta(\lambda)}\]
and $\frac{d\theta}{d\lambda}<0.$ Let 
\[T_\delta(\lambda)=R_{\frac{\delta\lambda}{2\pi}}T(\lambda),\]
then $T'(\lambda)=T[\varphi'](\lambda,0,T')=T_\delta(\lambda)^nT(\lambda)^n$ according to the $(\delta,n)$--padding $\varphi'$ of $\varphi.$ Introduce $H(t)=\exp(t\begin{bmatrix}0&\lambda\\-\lambda&0\end{bmatrix})=R_{\frac{t\lambda}{2\pi}}$ for $\lambda\in J$. For every $\kappa>0$,  $u(\lambda)=\mathbf{u}[T(\lambda)]$, for  $\delta$  sufficiently small we have for every $0\leq t<\delta$
\[d(H(t)\cdot u(\lambda),u(\lambda))<\kappa.\]

Since $J\subset\Omega(\varphi)$, for $\delta$ sufficiently small we have $|\Trace (T_\delta)|<2$ for every $\lambda\in J$. Let $B_\delta(\lambda)=\mathbf{B}[T_\delta(\lambda)],\theta_\delta(\lambda)=\mathbf{\Theta}[T_\delta(\lambda)]$, then clearly $B_\delta\to B, \theta_\delta\to\theta$ as $\delta\to 0$ as analytic functions of $\lambda\in J$. In particular
\begin{equation}
    \lim\limits_{\delta\to 0}\sup_{n}\sup_{\lambda\in J}\Vert B(\lambda)T'(\lambda)B(\lambda)^{-1}-R_{n(\theta(\lambda)+\theta_\delta(\lambda))}\Vert=0.
\end{equation}

For $0<\eta<\frac{1}{2}$, define $J_{\delta,\eta,n}\subset J$ as the set of $\lambda$ such that $2n(\theta(\lambda)+\theta_{\delta}(\lambda))$ is at least $\eta$ away from $\bbZ$. By the equidistribution property of Lemma \ref{lem:equidistributionLem}, together with Lemma \ref{lem.rotDerivative}, it follows that for sufficiently small $\delta$
\begin{equation}
|J_{\delta,\eta,n}|<(1-2\eta)|J|.
\end{equation}
Therefore, given $0<\eta<\frac{1}{2}$ and $\kappa>0$, if $\delta$ is chosen small enough, then for each $\lambda\in J_{\delta,\eta,n}$ we have $\lambda\in \Omega(\varphi')$ and 
\begin{equation}\label{eq.dist1}d(u'(\lambda,0),u(\lambda,0))\leq d(u'(\lambda),u_\delta(\lambda))+d(u_\delta(\lambda),u(\lambda))<\kappa,\end{equation}
where $u'(\lambda)=\mathbf{u}[T'(\lambda)]$. This follows from the fact that $T_\delta(\lambda)$ is close to $T(\lambda)$ whenever $\delta$ is sufficiently small.

Let us now estimate $d(u'(\lambda,t'),u(\lambda,t(t')))$ for some suitable choice of $t=t(t')$ and $t'\in[0,P']$.
For $0\leq j\leq n$, $T(\lambda)^ju'(\lambda,0)=u'(\lambda,a_j)$ and for $n\leq j\leq 2n$, $u'(\lambda,a_j)=T_\delta(\lambda)^{j-n}T(\lambda)^nu'(\lambda,0)$. Thus, we have
\begin{equation}\label{eq.dist2}d(u'(\lambda,a_j),u(\lambda,0))\leq \kappa.\end{equation}
This follows from \eqref{eq.dist1}, the closeness of $T_\delta(\lambda)$ to $T(\lambda)$, and that $T(\lambda)$ fixes  $u(\lambda)$. By the invariance of the hyperbolic distance, we obtain for $0\leq j\leq 2n-1$ and $t\in [a_j,a_j+P]$:
\begin{equation}\label{eq.dist3}
d(u'(\lambda,t),u(\lambda,t-a_j))<\kappa.
\end{equation}
Notice for $0\leq j\leq 2n-1$ and $t\in [a_j+P,a_j+P+\delta]$
\[H(a_j+P+\delta-t)\cdot u'(\lambda,t)=u'(\lambda,a_{j+1}).\]
Therefore, again by the invariance of the hyperbolic distance,
\begin{equation}\label{eq.dist4}\begin{aligned}d(u'(\lambda,t),u(\lambda))&=d(u'(\lambda,a_{j+1}),H(a_j+P+\delta-t)\cdot u(\lambda))\\
&\leq d(u'(\lambda,a_{j+1}),u(\lambda))+d(H(a_j+P+\delta-t)\cdot u(\lambda),u(\lambda))\\
&\leq 2\kappa.\end{aligned}\end{equation}
Following \eqref{eq.dist3} and \eqref{eq.dist4}, define 
\[
t(t') =
\begin{cases}
t-a_j, & t'\in [a_j,a_j+P], \\[6pt]
0, & t'\in [a_j+P,a_{j+1}].
\end{cases}
\]
We have shown that for $\lambda\in J_{\delta,\eta,n}$ and $t'\in[0,P'],$
\[d(u'(\lambda,t'),u(\lambda,t(t')))<2\kappa.\]

Therefore, for every $0<\eta<\frac{1}{2}$ and $\kappa>0$, if $\delta$ is sufficiently small and $n$ sufficiently large, we obtain 
\begin{equation}\label{eq.measureEst1}
    \begin{aligned}\mu_{\varphi'_{t'},C}(J_{\delta,\eta,n})&\geq e^{-2\kappa}\mu_{\varphi_{t},C}(J_{\delta,\eta,n})\\
    &\geq e^{-2\kappa}(\mu_{\varphi_t,C}(J)-\mu_{\varphi_t,C}(J\setminus J_{\delta,\eta,n})) \\
    &\geq \mu_{\varphi_t,C}(J)-2(\kappa+\eta)C|J|.
    \end{aligned}
\end{equation}
Since $\eta,\kappa$ are arbitrary, we may take them to be sufficiently small so that 
\begin{equation}\label{eq.dist5}\mu_{\varphi'_{t'},C}((-\infty,M])\geq \mu_{\varphi_t,C}(J)-\frac{\epsilon-\epsilon_0}{2}.\end{equation}
By the construction of $\varphi'$ and the choice of $t=t(t')$, we also have that for any $\epsilon_1>0$ and $C_1>0$, if $\delta$ is sufficiently small then 
\[\sup_{|s|\leq C_1}|\varphi'(t'+s)-\varphi(t+s)|<\epsilon_1.\]
Applying Lemma~\ref{lem:spectralCont}, it follows that for $\delta$ sufficiently small,
\[\mu_{\varphi'_{t'}}(-\infty,M]<\mu_{\varphi_t}(-\infty,M]+\frac{\epsilon-\epsilon_0}{2}.\]
Finally, combining \eqref{eq.uniformEst} with \eqref{eq.dist5}, we conclude that for every $t'\in[0,P']$,
\[\mu_{\varphi'_{t'}}(-\infty,M]\leq \mu_{\varphi'_{t'},C}(-\infty,M]+\epsilon.\]
This completes the proof.
\end{proof}
With all the preparatory lemmas in place, we proceed to the proof of Theorem \ref{thm:Main2}.
\begin{proof}[Proof of Theorem \ref{thm:Main2}]
The goal is to construct a sequence $P^{(n)}$ of periodic time change of solenoidal flow $F_t^{(n)}:S^{(n)}\to S^{(n)}$ and continuous functions $\Psi^{(n)}:S^{(n)}\to\bbC$ which admit projective limits that are sufficiently close to the lifts of each $(F^{(n)},\Psi^{(n)})$.
To initiate the inductive construction, take $P^{(0)}=1, S^{(0)}=\bbR/\bbZ,$ and let $F^{(0)}_t=F^{S^{(0)}}_t$ be the solenoidal flow on $S^{(0)}$. Choose $\Psi^{(0)}:\bbR/\bbZ\to\bbC$ to be a nonconstant smooth function which vanishes near $0$, and set $\kappa_0=1$. 
    
\emph{Step 1:} Choose $C_{j-1}$ sufficiently large so that $\Psi^{(j-1)}(F_t^{(j-1)}(0))$ is $(2^{1-j},C_{j-1},2^{j-1})$--uniform.
Note that the existence of such $C_{j-1}$ is guaranteed by Lemma \ref{lem:periodicUniform} for any periodic potential.

\emph{Step 2:} For sufficiently small $\kappa_{j-1}$, choose $(F^{(j)},\Psi^{(j)})$ that is $(2^{1-j'},C_{j'-1},2^{j'-1})$--uniform for all $j'\leq j-1$ and $(F^{(j)},\Psi^{(j)})$ is $\kappa_{j-1}$ close to a lift of $(F^{(j-1)},\Psi^{(j-1)})$ and $F^{(j)}$ is $(2^{j-1},F^{(j-1)})$--mixed; This step calls for the realization of $(\delta,n)$ padding as a continuous sampling along a flow by Lemma \ref{lem:realizeAsFlow}. The required $(\epsilon,C,M)$--uniform property is guaranteed by Lemma \ref{lem:stableUniform} while performing $(\delta,n)$--padding for sufficiently small $\delta$ and sufficiently large $n$. The stability of being $(\epsilon,C,M)$--uniform is essential here.
    
\emph{Step 3:} Let $0<\kappa_j<\kappa_{j-1}/2$ be sufficiently small so that if $(F,\varphi)$ is $2\kappa_j$ close to a lift of $(F^{(j)},\varphi^{(j)})$, then it is also $(2^{j-1},F^{(j-1)})$--mixed;
This follows from the stability of being $(N,F)$--mixed (Lemma \ref{lem:stabilityOfMixed}).
\medskip

Iterating the construction produces a sequence $(F^{(n)},S^{(n)},\Psi^{(n)})$. Let $(F,S,\Psi)$ be the projective limit, and define $\varphi_x(t)=\Psi(F_t(x))$ for any $x\in S$. Notice that by construction, $F$ is $(2^{j},F^{(j)})$--mixed for all $j$, hence $F$ is weakly mixing by Lemma \ref{lem:stabilityOfMixed}.
Moreover, since $\varphi(t)=\Psi(F_t x)$ is $(2^{-j},C_j,2^{j})$--uniform for any $j$, we compute
\[\mu_{\varphi}(-\infty,2^j]-\int_{-\infty}^{2^j}\frac{d\mu_{\varphi}}{d\lambda}d\lambda\leq \lim\limits_{k\to\infty}\left(\mu_{\varphi^{(k)}}(-\infty,2^j]-\mu_{\varphi^{(k)},C_j}(-\infty,2^j]\right)<2^{-j},\]
    where we take $\varphi^{(k)}(t)=\Psi^{(k)}(F^{(k)}_t (0)).$
\end{proof}

\section{Appendix}\label{sec:Construction}
In this appendix, we supplement the argument in Theorem \ref{thm:mainThm} by constructing a set $\sfE\subset\bbR$ such that the following properties are satisfied:\begin{enumerate}
    \item $\Omega=\bbC\setminus\sfE$ is a regular Widom domain
    \item DCT fails on $\Omega$
    \item $\sfE$ is closed unbounded in both sides, weakly homogeneous and moreover, if we denote $\sfE=\bbR\setminus\left(\cup_{j}(a_j,b_j)\right)$, then $\inf \{a_j:j\in J\}=-\infty,\sup\{b_j:j\in J\}=+\infty$ and 
    \[\sum_{j}\max\{1, |a_j|,|b_j|\}(b_j-a_j)<\infty.\]
\end{enumerate}

We first isolate two elementary perturbation lemmas that will be used
to modify Yuditskii's example \cite{Yuditskii2011DCTFails}.  We emphasize that the gaps will be
inserted inside a negative spectral band.  This choice will be useful
in the proof of persistence of the DCT failure.

Let $\sfE_0=[0,+\infty)\cup(\cup_{k\geq 1}[b_{k+1},a_k])$ be the set constructed in \cite{Yuditskii2011DCTFails} and let
$\Omega_0=\mathbb C\setminus\mathsf E_0$  be a regular Denjoy domain.  For a fixed pole $z_*=i$, write $g_0(z)=G_{\Omega_0}(z,i).$
Recall that, since $\Omega_0$ is symmetric with respect to $\mathbb R$,
for $x\in\mathbb R\setminus\mathsf E_0$ we have $G_{\Omega_0}(x,i)=G_{\Omega_0}(x,-i).$
Thus the maxima of $g_0$ on the real gaps give the sum in Widom condition \eqref{eq:WidomCond}. Our first lemma is on \emph{insertion of miscroscopic gaps}:
\begin{lemma}
\label{lem:microscopic-gaps}
Assume that $\Omega_0$ is a regular Widom domain and that
$\mathsf E_0$ is an Akhiezer--Levin set.  Suppose that for some
$x_0\in\mathbb R$ and $d_0>0$,
\[
[x_0-4d_0,x_0+4d_0]
   \subset\operatorname{int}\mathsf E_0.
\]
Let $d_n=2^{-n}d_0$ and let $
J_n^\pm =\left(x_0\pm d_n-\frac{\eta_n}{2}, x_0\pm d_n+\frac{\eta_n}{2} \right),$
where
\[
0<\eta_n\le\frac{d_n}{16}
\]
and
\begin{equation}
\label{eq:Beurling-summability}
\sum_{n\ge1}
\sqrt{\frac{\eta_n}{d_n}}<\infty.
\end{equation}
Define
\[
\mathsf E^\sharp
=
\mathsf E_0
\setminus
\bigcup_{n\ge1}(J_n^+\cup J_n^-),
\qquad
\Omega^\sharp=\mathbb C\setminus\mathsf E^\sharp.
\]

Then $\Omega^\sharp$ is a regular Widom domain and
$\mathsf E^\sharp$ is again an Akhiezer--Levin set.  Moreover,
$\mathsf E^\sharp$ is weakly homogeneous. 
If $c_n^\pm\in J_n^\pm$ is the critical point of
$G_{\Omega^\sharp}(\cdot,i)$ in the new gap $J_n^\pm$, then
\begin{equation}
\label{eq:new-critical-estimate}
G_{\Omega^\sharp}(c_n^\pm,i)
\le
C\sqrt{\frac{\eta_n}{d_n}},
\end{equation}
where $C$ is independent of $n$.
\end{lemma}

\begin{proof}
We divide the proof into several steps.

\medskip
\noindent
\emph{Step 1: opening one fixed interval preserves the
Akhiezer--Levin and Widom properties.}

Choose an interval \(L=(x_0-2d_0,x_0+2d_0)\) so that \(\overline L\subset\operatorname{int}\mathsf E_0
\) and all the intervals $J_n^\pm$ are contained in $L$.  Let $\Omega_L=\Omega_0\cup L,\, \mathsf E_L=\mathbb R\setminus\Omega_L,$
and denote $g_L(z)=G_{\Omega_L}(z,i).$
We first claim that there exist a compact neighborhood $Q$ of $L$
and a constant $C_0$ such that
\begin{equation}
\label{eq:compact-opening-comparison}
g_L(z)\le C_0g_0(z),
\qquad
z\in\Omega_0\setminus Q.
\end{equation}
Indeed, on $\Omega_0$ the function
\[
h(z)=g_L(z)-g_0(z)
\]
is positive and harmonic.  Its boundary values vanish on
$\mathsf E_0\setminus L$, while on $L$ they are bounded by
\[
M:=\sup_{x\in L}g_L(x)<\infty.
\]
Therefore, by harmonicity  of $h$ and vanishing on $\sfE_0\setminus L$
\begin{equation}
\label{eq:h-harmonic-measure}
0\le h(z)
\le
M\,\omega_{\Omega_0}(z,L),
\end{equation}
where $\omega_{\Omega_0}(z,L)$ denotes the harmonic measure of $L$
in $\Omega_0$.
We claim that
\begin{equation}
\label{eq:harmonic-measure-green-comparison}
\omega_{\Omega_0}(z,L)
\le C_1 g_0(z),
\qquad z\in\Omega_0\setminus Q.
\end{equation}
To see this, choose $Q$ to be a small rectangle around $L$ whose
intersection with the real axis is still contained in
$\operatorname{int}\mathsf E_0$.  On $\Gamma=\partial Q\cap\Omega_0$, both $\omega_{\Omega_0}(\cdot,L)$ and $g_0$ are positive harmonic
functions.  Away from the endpoints of $\Gamma$, compactness gives
a uniform bound for their ratio.  Near an endpoint of $\Gamma$ the
boundary is a straight interval, and the boundary Harnack principle
in a half-disc (compare e.g. \cite[Lemma 8]{Dahlberg1977}) gives the same bound up to the endpoint.  Hence
\[
\omega_{\Omega_0}(z,L)\le C_1g_0(z),
\qquad z\in\Gamma.
\]
On $\mathsf E_0\setminus Q$ both functions have zero boundary values.
Since $\mathsf E_0$ is Akhiezer--Levin, infinity is a regular
Dirichlet boundary point, and hence both functions tend to zero at
infinity.  Applying the maximum principle in
$\Omega_0\setminus\overline Q$ proves
\eqref{eq:harmonic-measure-green-comparison}.
Together with \eqref{eq:h-harmonic-measure}, this proves
\eqref{eq:compact-opening-comparison}.

By the Akhiezer--Levin criterion, we have
$\int_{\mathbb R}g_0(x)\,dx<\infty. $
Since $g_L$ is bounded on compact subsets of the real line and
\eqref{eq:compact-opening-comparison} holds outside $Q$, we obtain
\begin{equation}\label{eq:ALcond1}
\int_{\mathbb R}g_L(x)\,dx<\infty.
\end{equation}
It follows that $\mathsf E_L$ is again an Akhiezer--Levin set.

Let $(a_k,b_k)$ be any old gap in $\sfE_0$ that does not meet $Q$.  Then by \eqref{eq:compact-opening-comparison} and that $\sfE_0$ is a Widom set, we have
\[
\max_{x\in(a_k,b_k)}g_L(x)
\le
C_0
\max_{x\in(a_k,b_k)}g_0(x)<\infty.
\]
There are only finitely many old gaps meeting
$Q$, and the newly opened interval $L$ contributes only one further
critical value.  Hence
\[
\sum_{\nabla G_{\Omega_L}(c,i)=0}
G_{\Omega_L}(c,i)<\infty.
\]
It follows that $\Omega_L$ is also a Widom domain.

\medskip
\noindent
\emph{Step 2: comparison with the final domain.}

We now have that
\[
\Omega_0\subset\Omega^\sharp\subset\Omega_L.
\]
By the subordinate principle of Green functions,
\begin{equation}
\label{eq:green-monotonicity}
G_{\Omega^\sharp}(z,i)
\le
G_{\Omega_L}(z,i).
\end{equation}
By \eqref{eq:ALcond1},
\[
\int_{\mathbb R}
G_{\Omega^\sharp}(x,i)\,dx
\le
\int_{\mathbb R}
G_{\Omega_L}(x,i)\,dx
<\infty.
\]
It follows that $\mathsf E^\sharp$ is Akhiezer--Levin.

Furthermore, the contribution to the Parreau--Widom sum from all
the old gaps is finite, since for every old gap $(a_k,b_k)$,
\[
\max_{x\in(a_k,b_k)}
G_{\Omega^\sharp}(x,i)
\le
\max_{x\in(a_k,b_k)}
G_{\Omega_L}(x,i),
\]
and the corresponding sum for $\Omega_L$ is finite.
It remains only to estimate the newly created critical values.

\emph{Step 3: the new critical values.}

Let $x_n^\pm=x_0\pm d_n,
r_n=\frac{d_n}{8}
$ and let $
D_n^\pm=D(x_n^\pm,r_n).$ 
Because $d_{n+1}=d_n/2$ and $\eta_n\le d_n/16$, these discs can be
chosen pairwise disjoint.  Moreover, inside $D_n^\pm$ the only
inserted gap is $J_n^\pm$.
Let $g^\sharp(z)=G_{\Omega^\sharp}(z,i).$ On the part of the real diameter of $D_n^\pm$ belonging to
$\mathsf E^\sharp$, we have  $g^\sharp=0.$

On $\partial D_n^\pm$, by
\eqref{eq:green-monotonicity}, $g^\sharp\le g_L.$
Since all the discs $D_n^\pm$ lie in one fixed compact subset of
$\Omega_L$ and the pole $i$ is outside this compact set,
\[
M_0:=
\sup_n\sup_{\partial D_n^\pm}g_L<\infty.
\]

Let $c_n^\pm\in J_n^\pm$ be a point at which $g^\sharp$ attains its
maximum on the gap.  By the maximum principle,
\begin{equation}
\label{eq:green-harmonic-measure}
g^\sharp(c_n^\pm)
\le
M_0\,
\omega_{V_n^\pm}
(c_n^\pm,\partial D_n^\pm),
\end{equation}
where
\[
V_n^\pm
=
D_n^\pm
\setminus
\left(
[x_n^\pm-r_n,x_n^\pm-\eta_n/2]
\cup
[x_n^\pm+\eta_n/2,x_n^\pm+r_n]
\right).
\]

To estimate the harmonic measure in
\eqref{eq:green-harmonic-measure}, discard one of the two slits.
This enlarges the domain and hence increases the harmonic measure of
the outer circle.  After translation and scaling by $r_n$, the
remaining slit is a connected set joining the circle of radius
comparable to $\eta_n/r_n$ to the unit circle.  The Beurling
projection estimate \cite[Theorem 3.76]{Lawler2005} therefore gives
\[
\omega_{V_n^\pm}
(c_n^\pm,\partial D_n^\pm)
\le
C
\sqrt{\frac{\eta_n}{r_n}}
\le
C'
\sqrt{\frac{\eta_n}{d_n}}.
\]
Substitution in \eqref{eq:green-harmonic-measure} yields
\[
G_{\Omega^\sharp}(c_n^\pm,i)
\le
C''
\sqrt{\frac{\eta_n}{d_n}},
\]
which is \eqref{eq:new-critical-estimate}.
By \eqref{eq:Beurling-summability},
\[
\sum_{n\ge1}
\left(
G_{\Omega^\sharp}(c_n^+,i)
+
G_{\Omega^\sharp}(c_n^-,i)
\right)<\infty.
\]
Together with the already established summability of the critical
values in the old gaps, this proves that $\Omega^\sharp$ is Widom.

\emph{Step 4: weak homogeneity and regularity.}

All points of $\mathsf E^\sharp$ except $x_0$ are locally either
interior points of spectral intervals or endpoints of gaps.  Thus
for every such point $x$,
\[
\limsup_{\delta\downarrow0}
\frac{
|\mathsf E^\sharp\cap(x-\delta,x+\delta)|
}{2\delta}
\ge1/2.
\]

It remains to consider $x_0$.  Write
\[
\alpha_n=\frac{\eta_n}{d_n}.
\]
Condition \eqref{eq:Beurling-summability} implies
\[
\sum_n\alpha_n<\infty.
\]
If
\[
d_{N+1}<\delta\le d_N,
\]
then
\[
\left|
(x_0-\delta,x_0+\delta)
\setminus\mathsf E^\sharp
\right|
\le
2\sum_{n\ge N}\eta_n
=
2\sum_{n\ge N}\alpha_nd_n
\le
2d_N\sum_{n\ge N}\alpha_n
=o(\delta).
\]
Consequently,
\[
|\mathsf E^\sharp\cap(x_0-\delta,x_0+\delta)|
=
2\delta-o(\delta),
\]
and in particular
\[
\lim_{\delta\downarrow0}
\frac{
|\mathsf E^\sharp\cap(x_0-\delta,x_0+\delta)|
}{2\delta}
=1.
\]
Thus $\mathsf E^\sharp$ is weakly homogeneous.  Since $\mathsf E^\sharp$ is weakly homogeneous, every point
of $\mathsf E^\sharp$ is Dirichlet regular by the planar
Wiener criterion
\cite[Theorem~1.1, with $p=n=2$]{KM1994}.
\end{proof}

Next we show that the failure of DCT of \cite{Yuditskii2011DCTFails} is preserved in our new domain.
Moreover, we verify the gap length conditions \eqref{eq:finiteGapCond} and \eqref{eq:eq:gapLengthCond2}.
\begin{lemma}
\label{lem:DCT-defect-persists}
Let $\mathsf E_0$ be the set constructed in
\cite{Yuditskii2011DCTFails} and let $\Omega_0=\mathbb C\setminus\mathsf E_0.$
Choose the interval in Lemma~\ref{lem:microscopic-gaps} so that
\begin{equation}
\label{eq:negative-band}
[x_0-4d_0,x_0+4d_0]
\Subset
\operatorname{int}\mathsf E_0\cap(-\infty,0).
\end{equation}
Let $d_n=2^{-n}d_0, \eta_n=d_n^8$ and let $\mathsf E^\sharp$ and $\Omega^\sharp$ be obtained by inserting the microscopic gaps
\[
J_n^\pm
=
\left(
x_0\pm d_n-\frac{\eta_n}{2},
x_0\pm d_n+\frac{\eta_n}{2}
\right)
\]
as in Lemma~\ref{lem:microscopic-gaps}. Then DCT fails in $\Omega^\sharp$. Moreover, let
$\Phi(z)=\frac{1}{z-x_0}$ and define
\[
\mathsf E=\Phi(\mathsf E^\sharp)
\subset\widehat{\mathbb R}.
\]
Identifying $\Phi(x_0)=\infty$, the set $\mathsf E$ is a closed subset
of $\mathbb R$ whose gaps are unbounded in both directions.  If
\[
\mathbb R\setminus\mathsf E
=
\bigcup_j(a_j,b_j),
\]
then
\begin{equation}
\label{eq:weighted-gap-final}
\sum_j
\max\{1,|a_j|,|b_j|\}(b_j-a_j)<\infty.
\end{equation}
In addition, the conformally equivalent domain $\Omega=\mathbb C\setminus\mathsf E$ fails DCT.
\end{lemma}

\begin{proof}
We first prove that the DCT failure survives the microscopic gap
insertion.  We use the same function as in Yuditskii's proof of
\cite[Proposition~11]{Yuditskii2011DCTFails}.  Let
$c_1\in(a_1,b_1)$ be the normalization point chosen there and $t>0$, let
\begin{equation}
\label{eq:Yuditskii-witness}
F_\tau(z)
=
\cos(\tau\sqrt{c_1})-\cos(\tau\sqrt z),
\qquad
0<\tau\le t.
\end{equation}
The expression $\cos(\tau\sqrt z)$ is interpreted through its entire
power series and hence is an entire function of $z$.
Yuditskii proves that
\begin{equation}
\label{eq:Yuditskii-integrability}
\int_{\mathsf E_0}
\frac{|F_\tau(x)|}{|x-c_1|^2}\,dx<\infty.
\end{equation}
Since $\mathsf E^\sharp\subset\mathsf E_0,$ we immediately obtain
\begin{equation}
\label{eq:new-integrability}
\int_{\mathsf E^\sharp}
\frac{|F_\tau(x)|}{|x-c_1|^2}\,dx<\infty.
\end{equation}
It remains to verify that $F_\tau$ is of Smirnov class in
$\Omega^\sharp$.  Choose the branch of $\sqrt z$ in $
\mathbb C\setminus[0,\infty)$ for which $\operatorname{Im}\sqrt z>0.$
Since the microscopic surgery is performed entirely on the negative
half-axis, we still have $[0,\infty)\subset\mathsf E^\sharp,$ and hence $\Omega^\sharp\subset\mathbb C\setminus[0,\infty).$
Therefore, $q_\tau(z):=e^{i\tau\sqrt z}$ is a single-valued, nonvanishing bounded analytic function in
$\Omega^\sharp$, with
\[
|q_\tau(z)|
=
e^{-\tau\operatorname{Im}\sqrt z}
\le1.
\]

We claim that $q_\tau$ is outer in $\Omega^\sharp$.   Let $S$ be its
singular inner factor and set
\[
s(z)=-\log|S(z)|.
\]
Since $q_\tau$ has no zeros, there is no Blaschke factor, and
\begin{equation}
\label{eq:s-dominated}
0\le s(z)
\le
-\log|q_\tau(z)|
=
\tau\operatorname{Im}\sqrt z.
\end{equation}
Since \(S\) is inner, \(s\) has zero boundary values almost everywhere on each spectral interval. Local boundedness and the Poisson representation in a half-disc imply that \(s\) extends continuously by zero across the interior of each boundary interval. The same conclusion holds at  gap endpoints after straightening the slit by a square-root coordinate. Consequently, the singular measure has no support over any  finite boundary point that is either interior or endpoint of a gap.  It remains to treat the exceptional accumulation point \(x_0\) and infinity.

By \eqref{eq:negative-band}, the function $\sqrt z$ is analytic in a
full neighborhood of $x_0$, and the right-hand side of
\eqref{eq:s-dominated} is bounded there.  Since $x_0$ is a regular
boundary point by Lemma~\ref{lem:microscopic-gaps}, by the Martin representation $s(z)=\int K(z,\xi)d\chi(\xi)$, where $K(z,\xi)$ is the Martin kernel and $\chi$ is the harmonic measure.  If $\chi(\{x_0\})>0$, then $s(z)\to\infty$
as $z\to x_0$, contradicting \eqref{eq:s-dominated}.  Thus the
singular factor has no finite support.

A possible singular inner factor can therefore only be supported at
infinity.  By Lemma~\ref{lem:microscopic-gaps},
$\mathsf E^\sharp$ is Akhiezer--Levin, so its symmetric Martin
function $M^\sharp$ satisfies
\[
M^\sharp(iy)\sim \varkappa y,
\qquad y\to+\infty,
\]
for some $\varkappa>0$.  On the other hand,
\[
0\le s(iy)
\le
\tau\operatorname{Im}\sqrt{iy}
=
\frac{\tau}{\sqrt2}\sqrt y
=
o(M^\sharp(iy)).
\]
Hence no positive Martin mass at infinity can occur, and therefore $S\equiv1.$ Thus $q_\tau$ is outer.

It follows that $q_\tau^{-1}$ belongs to the Smirnov class of
$\Omega^\sharp$.  Since
\[
\cos(\tau\sqrt z)
=
\frac12
\left(
q_\tau(z)+q_\tau(z)^{-1}
\right),
\]
the function $F_\tau$ belongs to the Smirnov class of
$\Omega^\sharp$.
Combining this with \eqref{eq:new-integrability}, the function
$F_\tau$ belongs to the same $E^1_0$ class used in Yuditskii's proof
of Proposition~5.2.

Since $F_\tau$ is entire, its upper and lower boundary values coincide
on $\mathsf E^\sharp$.  Therefore
\begin{equation}
\label{eq:DCT-integral-zero}
\frac{1}{2\pi i}
\int_{\mathsf E^\sharp}
\left[
\frac{F_\tau(x+i0)}{(x-c_1)^2}
-
\frac{F_\tau(x-i0)}{(x-c_1)^2}
\right]\,dx
=0.
\end{equation}
On the other hand,
\[
F_\tau(c_1)=0
\]
and
\[
F_\tau'(c_1)
=
\frac{\tau}{2\sqrt{c_1}}
\sin(\tau\sqrt{c_1})\neq0.
\]
Hence
\begin{equation}
\label{eq:nonzero-residue}
\operatorname*{Res}_{z=c_1}
\frac{F_\tau(z)}{(z-c_1)^2}\,dz
=
F_\tau'(c_1)
\neq0.
\end{equation}
If DCT held in $\Omega^\sharp$, the boundary integral
\eqref{eq:DCT-integral-zero} would equal the residue
\eqref{eq:nonzero-residue}, which is impossible.  Thus DCT fails in
$\Omega^\sharp$.

We now verify the weighted gap condition after applying the M\"obius
map
\[
\Phi(z)=\frac{1}{z-x_0}.
\]

There are two families of gaps of $\mathsf E$: the images of the
microscopic gaps $J_n^\pm$ and the images of the original gaps of $\sfE_0$. For the positive microscopic gap
\[
J_n^+
=
\left(
x_0+d_n-\frac{\eta_n}{2},
x_0+d_n+\frac{\eta_n}{2}
\right),
\]
we have
\[
\Phi(J_n^+)
=
\left(
\frac{1}{d_n+\eta_n/2},
\frac{1}{d_n-\eta_n/2}
\right)
=:(a_n^+,b_n^+).
\]
Therefore
\begin{equation}
\label{eq:new-gap-length-plus}
b_n^+-a_n^+
=
\frac{\eta_n}
{d_n^2-\eta_n^2/4},
\end{equation}
and
\[
\max\{|a_n^+|,|b_n^+|\}
=
\frac{1}{d_n-\eta_n/2}.
\]
Consequently,
\begin{equation}
\label{eq:new-gap-weight-plus}
\max\{|a_n^+|,|b_n^+|\}(b_n^+-a_n^+)
=
\frac{\eta_n}
{(d_n-\eta_n/2)^2(d_n+\eta_n/2)}.
\end{equation}

Similarly,
\[
\Phi(J_n^-)
=
\left(
-\frac{1}{d_n-\eta_n/2},
-\frac{1}{d_n+\eta_n/2}
\right)
=:(a_n^-,b_n^-),
\]
and the weighted contribution of this gap is identical to that in
\eqref{eq:new-gap-weight-plus}.

Since $\eta_n=d_n^8$ and $d_n\to0$, after decreasing $d_0$ if necessary we may assume $\eta_n\le\frac{d_n}{2}$ for every $n$.  Hence
\[
d_n-\frac{\eta_n}{2}
\ge\frac34d_n,
\qquad
d_n+\frac{\eta_n}{2}\ge d_n,
\]
and therefore
\begin{equation}
\label{eq:new-gap-weight-bound}
\max\{1,|a_n^\pm|,|b_n^\pm|\}
(b_n^\pm-a_n^\pm)
\le
C\frac{\eta_n}{d_n^3}
=
Cd_n^5.
\end{equation}
Since $d_n=2^{-n}d_0,$
we obtain
\[
\sum_{n\ge1}d_n^5
=
d_0^5\sum_{n\ge1}2^{-5n}<\infty.
\]
Thus
\begin{equation}
\label{eq:new-gaps-weighted-summable}
\sum_{n\ge1}
\max\{1,|a_n^\pm|,|b_n^\pm|\}
(b_n^\pm-a_n^\pm)
<\infty.
\end{equation}

We next consider the images of the original gaps of
$\mathsf E_0$.  Since $x_0$ belongs to the interior of a spectral
band, there exists $\rho>0$ such that $(x_0-\rho,x_0+\rho)\subset\mathsf E_0.$ Hence every original gap $(A_k,B_k)$ of $\mathsf E_0$ satisfies
$\operatorname{dist}((A_k,B_k),x_0)\ge\rho.$
Its image under $\Phi$ is therefore contained in $\left[-\frac1\rho,\frac1\rho\right]$.
If $(\alpha_k,\beta_k)=\Phi((A_k,B_k))$, then $\max\{|\alpha_k|,|\beta_k|\}\le\frac1\rho.$
Since the transformed old gaps are pairwise disjoint and all lie in
the bounded interval
$\left[-\frac1\rho,\frac1\rho\right],$
we have
\[
\sum_k(\beta_k-\alpha_k)\le\frac{2}{\rho}.
\]
Therefore
\begin{equation}
\label{eq:old-gaps-weighted}
\begin{aligned}
\sum_k
\max\{1, |\alpha_k|,|\beta_k|\}
(\beta_k-\alpha_k)
&\le
(1+\frac1\rho)
\sum_k(\beta_k-\alpha_k)\\
&\le
\frac{2}{\rho}+\frac{2}{\rho^2}<\infty.
\end{aligned}
\end{equation}

Combining \eqref{eq:new-gaps-weighted-summable} and
\eqref{eq:old-gaps-weighted}, we obtain
\[
\sum_j
\max\{1,|a_j|,|b_j|\}(b_j-a_j)<\infty.
\]

Moreover,
\[
a_n^+,b_n^+\to+\infty,
\qquad
a_n^-,b_n^-\to-\infty,
\]
so the gaps of $\mathsf E$ are unbounded in both directions.

Finally, $\Phi$ is a conformal equivalence between
$\Omega^\sharp$ and $\Omega$.  Since DCT is invariant under conformal
equivalence when formulated for Smirnov differentials, failure of DCT in $\Omega^\sharp$ implies failure of DCT in $\Omega$. This completes the proof.
\end{proof}

\noindent\textbf{AI usage statement:} The idea of this construction belongs to the authors. We stress tested and tighten up the scales in the construction in the Appendix by ChatGPT 6 Astra. The example only serves the purpose that both sets of spectrum $\sfE$ in the statement of Theorem \ref{thm:mainThm} are nonempty. The use of AI tools is only limited to this appendix.

%\bibliographystyle{abbrv}
% \bibliographystyle{alpha}

% \bibliography{MATH522.bib}

\newcommand{\etalchar}[1]{$^{#1}$}

\end{document}